\documentclass[11pt]{amsart}
\usepackage[a4paper,margin=1in]{geometry}
\usepackage{amsmath,amssymb,amsfonts,amsthm,mathtools,mathrsfs,enumitem,microtype,booktabs}
\usepackage[colorlinks=true,linkcolor=blue,citecolor=blue,urlcolor=blue]{hyperref}
\numberwithin{equation}{section}
\allowdisplaybreaks

\newtheorem{theorem}{Theorem}[section]
\newtheorem{proposition}[theorem]{Proposition}
\newtheorem{lemma}[theorem]{Lemma}

\newtheorem{kernelestimate}[theorem]{Kernel estimate}

\theoremstyle{definition}
\newtheorem{definition}[theorem]{Definition}
\newtheorem{remark}[theorem]{Remark}

\newcommand{\Fp}{\mathbb F_p}
\newcommand{\E}{\mathbb E}
\newcommand{\A}{\mathbb A}

\title[Polynomial corners beyond distinct degrees]{Polynomial corners in finite fields beyond the distinct-degree case}

\author{Ji Li}
\address{School of Mathematical and Physical Sciences, Macquarie University, NSW 2109, Australia}
\email{ji.li@mq.edu.au}

\author{Chun-Yen Shen}
\address{Department of Mathematics, National Taiwan University, Taipei 10617, Taiwan}
\email{cyshen@math.ntu.edu.tw}

\author{Tuyen Trung Truong}
\address{Department of Mathematics, University of Oslo, P.O. Box 1053, Blindern, 0316 Oslo, Norway}
\email{tuyentt@math.uio.no}

\author{Liangchuan Wu}
\address{School of Mathematical Sciences, Anhui University, Hefei, 230601, P.R.~China}
\email{wuliangchuan@ahu.edu.cn}


\begin{document}

\begin{abstract}

We prove a quantitative polynomial Roth theorem for corners in \(\mathbb F_p^2\) for arbitrary pairs of linearly independent polynomials. More precisely, given a positive integer $d$, there are constants  $p_0$ and $C$ (depending only on $p$) so that for every $ p>p_0$,  if  polynomials \(\phi_1,\phi_2\in \mathbb \mathbb{F}_p [y]\) are of degree $\leq d$ vanishing at $0$ and are not linearly dependent, then every \(A\subset\mathbb F_p^2\) with $
|A|\ge C p^{2-1/14} $ contains a nontrivial corner
$$
(x_1,x_2),\qquad
(x_1+\phi_1(y),x_2),\qquad
(x_1,x_2+\phi_2(y))
$$
for some \(y\in\mathbb F_p^\times\). This improves the estimate $p^{2-1/16}$ of Han--Lacey--Yang and removes the distinct-degree restriction from their quantitative theorem.

%

The main obstruction is the equal-degree resonant case, where the
Jacobian argument of Han--Lacey--Yang degenerates.  We adjoin the
frequency-independent part of the phase to form an augmented map
\(\widetilde F:W\to\mathbb A^3\) from the correlation threefold.  We
prove that this map is generically finite on every top-dimensional
geometric component and has no two-dimensional fibre.  Using the
associated Artin--Schreier sheaf and Katz--Laumon estimates for Fourier transform of
perverse sheaves, we obtain square-root cancellation outside an
algebraic exceptional set of dimension at most one and uniformly bounded
degree.  A separate curve-sum argument gives uniform control on the
exceptional set.  An \(\ell^2\) matrix estimate adapted to such sets
completes the resonant case.

\end{abstract}
\maketitle

\section{Introduction}

Polynomial configurations in subsets of finite fields provide a natural
quantitative setting for questions originating in Roth's and Szemer\'edi's
theorems.  In several variables, even three-point configurations can exhibit
genuinely multidimensional phenomena, since different polynomial shifts may
act in different coordinate directions.

Let $p$ be a prime, and let $\phi_1,\phi_2\in\mathbb Z[y]$ satisfy
$
  \phi_1(0)=\phi_2(0)=0.
$

For $A\subset\Fp^2$, a \emph{polynomial corner} in $A$ is a configuration
$$
  (x_1,x_2),\qquad
  (x_1+\phi_1(y),x_2),\qquad
  (x_1,x_2+\phi_2(y))
$$
contained in $A$, where $x=(x_1,x_2)\in\Fp^2$ and
$y\in\Fp^\times$.  Since the two shifts act in different coordinate
directions, the distribution of $A$ along both rows and columns plays an
essential role.

The qualitative background begins with the polynomial extensions of
Szemer\'edi's theorem of Bergelson--Leibman
\cite{BergelsonLeibman}.  Quantitative polynomial Roth-type theorems over
finite fields were developed in several settings by Bourgain--Chang
\cite{BourgainChang}, Peluse \cite{Peluse}, and Dong--Li--Sawin
\cite{DLS}.  Kuca obtained power-saving estimates for multidimensional
polynomial configurations, first for families of distinct degrees
\cite{KucaDistinct}, and later for fixed linearly independent integer
polynomials without that restriction \cite{Kuca}.

For the two-coordinate corner problem, Han--Lacey--Yang \cite{HLY} proved
that if $\phi_1$ and $\phi_2$ are linearly independent and either have
distinct degrees or are both quadratic, then every
$$
  A\subset\Fp^2,
  \qquad
  |A|\gg p^{2-1/16},
$$
contains a nontrivial polynomial corner.  More recently, Lim \cite{Lim}
obtained an asymptotic counting formula for fixed rational functions
$P,Q\in\mathbb Q(t)$ with $1,P,Q$ linearly independent over $\mathbb Q$,
and deduced the explicit density exponent $1/122880$.

Our result improves the Han--Lacey--Yang \cite{HLY}  exponent and removes the
distinct-degree restriction altogether.

\begin{theorem}\label{thm:main}
For every positive integer $d>0$, there exist constants
$C=C(d)>0$ and $p_0=p_0(d)>0$ and  with the following property.  If polynomials $\phi_1,\phi_2\in\mathbb{F}_p[y]$ of degree $\leq d$ vanishing at $0$ and are not linearly dependent, 
then every
$A\subset\Fp^2$ with
$$
  |A|\geq Cp^{2-1/14}
$$
contains a polynomial corner with parameter $y\in\Fp^\times$.
\end{theorem}

The proof has two main new ingredients.  The first is analytic and is
responsible for the exponent $1/14$.  The second is geometric and resolves
the equal-degree resonant case, where the non-degeneracy mechanism in the
previous correlation estimate breaks down. 

\smallskip
{\bf The fibre-energy mechanism.} \ \ 
We use the Han--Lacey--Yang decomposition of the normalized corner count
into a main term, two partially degenerate terms, and a fully
non-degenerate bilinear term $J_3$.  Their $TT^*$ argument reduces the
estimate for $J_3$ to an $\ell^2$-operator bound for a one-dimensional
correlation matrix associated with each nonzero frequency shift
$$
  H=(H_1,H_2)\in\Fp^2.
$$

The required operator norm is $O(p^{-1/2})$, giving a $p^{-1/4}$ gain in
the estimate for $J_3$.

Our quantitative improvement comes from retaining the row and column fibre
energies that arise before the final $L^4$ majorization.  Let
$f=\mathbf 1_A$, $\delta=|A|/p^2$, and define
$$
  r(x_2)=\E_{x_1}f(x_1,x_2),
  \qquad
  c(x_1)=\E_{x_2}f(x_1,x_2).
$$

For the centred functions $f-r$ and $f-c$, set
$$
  V_R=\E_{x_2}\bigl(r(x_2)(1-r(x_2))\bigr)^2,
  \qquad
  V_C=\E_{x_1}\bigl(c(x_1)(1-c(x_1))\bigr)^2,
  \qquad
  V=V_RV_C.
$$
The fully non-degenerate contribution satisfies
$$
  \operatorname{Err}_{J_3}
  \lesssim
  p^{-1/2}\delta^{3/2}
  +
  p^{-1/8}\delta^{1/2}V^{1/4}.
$$
The same fibre structure also strengthens the main term
$$
  M=\E_{x_1,x_2}f(x_1,x_2)c(x_1)r(x_2).
$$
Two complementary bipartite-graph inequalities give
$$
  M\geq\delta^3
  \qquad\text{and}\qquad
  M\geq R_2C_2,
$$
where $R_2=\E r^2$ and $C_2=\E c^2$.  Since
$V_R\leq R_2$ and $V_C\leq C_2$,
$$
  M\geq\max(\delta^3,V).
$$

It follows that
$$
  p^{-1/8}\delta^{1/2}V^{1/4}
  \leq
  p^{-1/8}\delta^{-7/4}M,
$$
so this error is absorbed once
$
  \delta\gtrsim p^{-1/14}.
$

The remaining error terms are smaller in the same range.

The point is that concentration of $A$ along rows and columns enlarges the
fibre-energy error, but simultaneously forces the main term to be larger.
Keeping this information on both sides of the argument, rather than
replacing it by a global norm estimate, produces the exponent $1/14$.

\smallskip
{\bf The equal-degree resonance.} \ \
The deeper obstruction occurs when
$$
  \deg\phi_1=\deg\phi_2=d>2.
$$
Write
$$
  \phi_1(y)=ay^d+\cdots,
  \qquad
  \phi_2(y)=by^d+\cdots.
$$

Note that Theorem \ref{thm:main} can be also equivalently stated in the form: if $\phi_1,\phi_2\in \mathbb{Z}[y]$ are such that their base change to $\mathbb{F}_p$ of degree $\leq d$ and their base change to $\mathbb{F}_p$ are not linearly dependent, then the conclusion of the theorem holds. This second statement allows us to work with a universal family of polynomials defined over $\mathbb{Z}$, which will be  extensively used for arguments in Section 6. 

The top-degree Jacobian calculation in
\cite[Appendix~A]{HLY} gives the required correlation estimate whenever
$$
  aH_1\neq bH_2.
$$

When $aH_1=bH_2$, however, the coefficient supplying the required rank
vanishes.  These resonant shifts form a one-parameter family.

Three-dimensional exponential sums arising from the same $TT^*$ mechanism
already appear in Dong--Li--Sawin \cite{DLS} and in the Han--Lacey--Yang
analysis \cite{HLY}.  Their pointwise estimates rely on geometric
non-degeneracy conditions for the leading constraint and leading
homogeneous forms.  In the equal-degree resonant regime these conditions
degenerate, so the previous argument cannot simply be extended.

Set
$$
  \alpha=\frac ba,
  \qquad
  D=\phi_2-\alpha\phi_1.
$$

Then $0<\deg D<d$, and every nonzero resonant shift has the form
$$
  H=(\alpha\lambda,\lambda),
  \qquad
  \lambda\in\Fp^\times.
$$

The corresponding correlation is reduced to an exponential sum on the
three-dimensional hypersurface
$$
  W=
  \left\{
    (y_1,y_2,y_3,y_4)\in\mathbb A^4:
    \phi_1(y_1)-\phi_1(y_2)
    =
    \phi_1(y_3)-\phi_1(y_4)
  \right\},
$$
of the form
$$
  S_\lambda(m,m')
  =
  \sum_{w\in W(\Fp)}
  \psi\!\bigl(f_0(w)+mf_1(w)+m'f_2(w)\bigr).
$$

The key geometric idea is to promote the frequency-independent part $f_0$
of the phase to an additional coordinate and consider
$$
  \widetilde F=(f_0,f_1,f_2):
  W\longrightarrow\mathbb A^3.
$$

We prove that $\widetilde F$ is generically finite on every
three-dimensional component of $W$ and contracts no surface to a point.
Katz--Laumon theory \cite{KatzLaumon} then gives a bounded-degree
exceptional locus
$$
  \mathcal Z_\lambda\subset\mathbb A^2_{m,m'},
  \qquad
  \dim\mathcal Z_\lambda\leq1,
$$
such that
$$
  |S_\lambda(m,m')|
  \lesssim_d p^{3/2}
  \qquad
  \bigl((m,m')\notin\mathcal Z_\lambda(\Fp)\bigr).
$$

To control the correlation matrix on the exceptional locus, we prove
separately the global estimate
$$
  |S_\lambda(m,m')|
  \lesssim_d p^2
  \qquad
  \text{for all }(m,m')\in\Fp^2.
$$

This follows by decomposing into fibres of
$\phi_1(x)-\phi_1(y)$ and combining Stein's irreducibility theorem with the
Bombieri--Weil bound on the resulting curve sums.  Thus the generic
$p^{3/2}$ estimate and the global $p^2$ estimate come from different
geometric mechanisms. This gives an estimate whose implied constants (i.e. $p_0,C$) may depend on the pair $\phi _1,\phi _2\in \mathbb{Z}[y]$, and not only on the degree $d$. To obtain a uniform estimate depending only on the degree $d$, we apply the ideas in \cite{BonolisKowalskiWoo} to the universal family parametrizing all such pairs $\phi _1,\phi _2\in \mathbb{Z}[y]$. 

After normalization, the correlation kernel is
$O_d(p^{-3/2})$ away from a bounded-complexity exceptional set and
$O_d(p^{-1})$ on that set.  Bounded-degree curves meet each row and column
in only $O_d(1)$ points unless they contain a vertical or horizontal
component, and those components are handled directly in $\ell^2$.
Schur's test then yields the required
$$
  O_d(p^{-1/2})
$$
operator norm.

The augmented-map mechanism may be useful more generally in resonant
correlation problems where the usual leading-form or Jacobian criterion
degenerates: a frequency-independent part of the phase can sometimes be
promoted to an additional coordinate, allowing generic Fourier cancellation
to be combined with a separate estimate on the exceptional locus.

The paper is organized as follows.
Section~\ref{sec:HLY-reduction} recalls the Han--Lacey--Yang reduction and
isolates the correlation-matrix estimate.
Section~\ref{sec:refined-fibre-energy} proves the refined fibre-energy bound
for $J_3$, and Section~\ref{sec:threshold} derives the exponent $1/14$.
Section~\ref{sec:admissible} treats bounded-complexity exceptional loci and
reduces the equal-degree resonant case to two exponential-sum estimates.
Section~\ref{sec:replacement} proves these estimates.  Finally,
Section~\ref{sec:exponential-sum-preliminaries} records the
algebro-geometric and curve-sum tools used in the proof.

\section{The reduction to the Han--Lacey--Yang  kernel estimate}
\label{sec:HLY-reduction}

In this section we recall the portion of the Han--Lacey--Yang decomposition
that will be used below, while fixing our Fourier normalization and the
associated kernel notation.  The only genuinely non-degenerate estimate in the
argument is isolated in Kernel Estimate~\ref{est:HLY-kernel}.  Sections
\ref{sec:admissible} and \ref{sec:replacement} will later establish this
estimate in the range required for Theorem~\ref{thm:main}.

Throughout the paper, we use the additive character
$\psi(t)=e_p(t):=e^{2\pi i t/p}$ for $t\in\Fp$.
Thus the normalized Fourier transform of a function
$f:\Fp^2\to\mathbb C$ is
$$
  \widehat f(\xi)
  =
  \E_{x\in\Fp^2}f(x)\psi(-x\cdot\xi),
  \qquad \xi\in\Fp^2.
$$
With this normalization, Plancherel's identity takes the form
$$
  \|f\|_{L^2(\Fp^2)}^2
  :=
  \E_{x\in\Fp^2}|f(x)|^2
  =
  \sum_{\xi\in\Fp^2}|\widehat f(\xi)|^2.
$$

For functions $f_1,f_2:\Fp^2\to\mathbb C$, define
$$
  A(f_1,f_2)(x)
  =
  \E_{y\in\Fp}
  f_1(x_1+\phi_1(y),x_2)
  f_2(x_1,x_2+\phi_2(y)).
$$
For a function $f:\Fp^2\to\mathbb C$, we also define the two coordinate
projections
$$
  P_1f(x_1,x_2)
  =
  \E_{u\in\Fp}f(u,x_2),
  \qquad
  P_2f(x_1,x_2)
  =
  \E_{v\in\Fp}f(x_1,v).
$$
Thus $P_1f$ is constant in the first coordinate, while $P_2f$ is
constant in the second coordinate.

The fully non-degenerate bilinear term is
$$
  J_3(g_1,g_2)(x)
  =
  \E_{y\in\Fp}
  g_1(x_1+\phi_1(y),x_2)
  g_2(x_1,x_2+\phi_2(y)),
$$
where the functions $g_1$ and $g_2$ satisfy
$\E_{x_1}g_1(x_1,x_2)=0$ for every $x_2\in\Fp$ and
$\E_{x_2}g_2(x_1,x_2)=0$ for every $x_1\in\Fp$.

\begin{kernelestimate}\label{est:HLY-kernel}
For every nonzero Han--Lacey--Yang frequency shift
$H=(H_1,H_2)\in\Fp^2\setminus\{0\}$,
the non-degenerate difference kernel arising in the $TT^*$-analysis of
$J_3$, made explicit in Lemma~\ref{lem:HLY-matrix-reduction}, satisfies
the operator bound
$$
  \|T_H\|_{\ell^2(\Fp)\to\ell^2(\Fp)}
  \lesssim p^{-1/2}.
$$
Equivalently, the argument of Han--Lacey--Yang, Lemma~4.3, yields the
corresponding $p^{-1/4}$ gain in the estimate for $J_3$.

Throughout Sections~2--5, the letter $H$ denotes this two-dimensional
Han--Lacey--Yang frequency shift.  The scalar resonant parameter introduced later in
Section~\ref{sec:admissible} will be denoted by $\lambda$.
\end{kernelestimate}

\begin{remark}\label{rem:two-shifts}
The distinction between $H$ and $\lambda$ will be important later.
The parameter $H=(H_1,H_2)\in\Fp^2$ is the frequency difference produced
by the Han--Lacey--Yang $TT^*$-argument.  By
contrast, after the equal-degree resonant reduction in
Section~\ref{sec:admissible}, we write
{$H=(\alpha\lambda,\lambda)$ with
$\lambda=H_2\in\Fp^\times$}.  For fixed $\lambda$, the
remaining matrix variables in that section are $m,m'\in\Fp$.
\end{remark}

\begin{lemma}\label{lem:HLY-matrix-reduction}
Let
$
  K(a,b)
  =
  \E_{y\in\Fp}
  \psi\!\bigl(a\phi_1(y)+b\phi_2(y)\bigr),
  \ 
  \widetilde K(a,b)
  =
  K(a,b)\mathbf 1_{b\neq0}.
$
For a nonzero Han--Lacey--Yang shift $H=(H_1,H_2)\in\Fp^2\setminus\{0\}$,
define
$$
  \Delta_H\widetilde K(n,m)
  =
  \widetilde K(n,m)\,
  \overline{\widetilde K(n-H_1,m+H_2)}.
$$
Define the correlation matrix
$$
  C_H(m,m')
  =
  \sum_{n\in\Fp}
  \Delta_H\widetilde K(n,m)\,
  \overline{\Delta_H\widetilde K(n,m')},
$$
and let $T_H$ be the corresponding operator on $\ell^2(\Fp)$:
$$
  (T_HF)(m)
  =
  \sum_{m'\in\Fp}C_H(m,m')F(m').
$$
Then the estimate required in \cite[Lemma~4.3]{HLY} is equivalent,
by duality and the $TT^*$-argument, to
$$
  \|T_H\|_{\ell^2(\Fp)\to\ell^2(\Fp)}
  \lesssim p^{-1/2}.
$$

Moreover, away from the vertical and horizontal lines on which one of the
four factors of $\widetilde K$ vanishes, the matrix entry $C_H(m,m')$
has the form
\begin{equation}
  C_H(m,m')
  =
  p^{-3}
  \sum_{\substack{y=(y_1,y_2,y_3,y_4)\in\Fp^4\\G(y)=0}}
  \psi\!\bigl(\mathcal H_{H,m,m'}(y)\bigr),
  \tag{2.1}\label{eq:HLY-pminus3}
\end{equation}
where
$$
  G(y)
  =
  \phi_1(y_1)-\phi_1(y_2)-\phi_1(y_3)+\phi_1(y_4),
$$
and
\begin{align*}
  \mathcal H_{H,m,m'}(y)
  ={}&
  H_1\bigl(\phi_1(y_2)-\phi_1(y_4)\bigr) 
  +
  m\bigl(\phi_2(y_1)-\phi_2(y_2)\bigr) \\
  &+
  m'\bigl(\phi_2(y_4)-\phi_2(y_3)\bigr) 
  +
  H_2\bigl(\phi_2(y_4)-\phi_2(y_2)\bigr).
\end{align*}
Changing the convention for complex conjugation replaces $H$ by
$-H$ and conjugates the resulting matrix.  Consequently, the relevant
operator norm estimates are unchanged.
\end{lemma}

\begin{proof}
Consider the bilinear form arising in \cite[Lemma~4.3]{HLY}:
$$
  \mathcal B_H(F_1,F_2)
  =
  \sum_{n,m\in\Fp}
  F_1(n)F_2(m)\Delta_H\widetilde K(n,m).
$$
Viewing this as a bilinear form in $F_1$ and $F_2$, duality in the
$n$-variable and the usual $TT^*$-argument produce the correlation
matrix
$$
  C_H(m,m')
  =
  \sum_{n\in\Fp}
  \Delta_H\widetilde K(n,m)\,
  \overline{\Delta_H\widetilde K(n,m')}.
$$
Thus an operator bound
$$
  \|T_H\|_{\ell^2\to\ell^2}
  \lesssim p^{-1/2}
$$
implies
$$
  |\mathcal B_H(F_1,F_2)|
  \lesssim
  p^{-1/4}\|F_1\|_{\ell^2}\|F_2\|_{\ell^2}.
$$
Conversely, this bilinear estimate is precisely the form used in the Han--Lacey--Yang
proof of the fully non-degenerate estimate.

We next verify the normalization and the phase in
\eqref{eq:HLY-pminus3}.  By definition,
\begin{align*}
  \Delta_H\widetilde K(n,m)
  ={}&
  \widetilde K(n,m)
  \overline{\widetilde K(n-H_1,m+H_2)}.
\end{align*}
Ignoring for the moment the cut-offs in the second variables, expansion of
the two normalized kernels gives
\begin{align*}
  \Delta_H K(n,m)
  =
  p^{-2}
  \sum_{y_1,y_2\in\Fp}
  \psi\Bigl(
    n\phi_1(y_1)+m\phi_2(y_1) 
    -(n-H_1)\phi_1(y_2)
    -(m+H_2)\phi_2(y_2)
  \Bigr).
\end{align*}
Similarly,
\begin{align*}
  \overline{\Delta_H K(n,m')}
  =
  p^{-2}
  \sum_{y_3,y_4\in\Fp}
  \psi\Bigl(
    -n\phi_1(y_3)-m'\phi_2(y_3) 
    +(n-H_1)\phi_1(y_4)
    +(m'+H_2)\phi_2(y_4)
  \Bigr).
\end{align*}
Multiplying these expressions and summing over $n\in\Fp$, we obtain
$$
  p^{-4}
  \sum_{y_1,y_2,y_3,y_4\in\Fp}
  \sum_{n\in\Fp}
  \psi\!\left(
    nG(y)+\mathcal H_{H,m,m'}(y)
  \right),
$$
where
$$
  G(y)
  =
  \phi_1(y_1)-\phi_1(y_2)-\phi_1(y_3)+\phi_1(y_4).
$$
The orthogonality relation
$$
  \sum_{n\in\Fp}\psi(nt)
  =
  \begin{cases}
    p,&t=0,\\
    0,&t\neq0,
  \end{cases}
$$
therefore restricts the sum to the hypersurface $G(y)=0$.  Since the
four normalized kernels contribute $p^{-4}$, while the $n$-sum
contributes a factor $p$, the resulting normalization is $p^{-3}$.

Collecting the terms that are independent of $n$ gives
\begin{align*}
  \mathcal H_{H,m,m'}(y)
  ={}&
  H_1\bigl(\phi_1(y_2)-\phi_1(y_4)\bigr) 
  +
  m\bigl(\phi_2(y_1)-\phi_2(y_2)\bigr) \\
  &+
  m'\bigl(\phi_2(y_4)-\phi_2(y_3)\bigr) 
  +
  H_2\bigl(\phi_2(y_4)-\phi_2(y_2)\bigr),
\end{align*}
which proves \eqref{eq:HLY-pminus3}.

Finally, the cutoff $\widetilde K(a,b)=K(a,b)\mathbf 1_{b\neq0}$ forces
the four second-frequency variables $m,m+H_2,m',m'+H_2$ to be nonzero.
Hence the difference between the full correlation matrix
and the matrix represented by \eqref{eq:HLY-pminus3} is supported on the
four lines $m=0$, $m=-H_2$, $m'=0$, and $m'=-H_2$.
These are boundedly many vertical and horizontal lines in the
$(m,m')$-plane.
\end{proof}

{

We next record the exact decomposition of the corner-counting form.  Let
$$
  \Lambda(f)
  =
  \E_{\substack{x\in\Fp^2\\y\in\Fp}}
  f(x)
  f(x_1+\phi_1(y),x_2)
  f(x_1,x_2+\phi_2(y)).
$$
Define $r(x_2)=\E_{x_1}f(x_1,x_2)$ and
$c(x_1)=\E_{x_2}f(x_1,x_2)$, and let $u=f-r$ and $v=f-c$.
Thus $u$ has zero average in the first coordinate and $v$ has zero average
in the second coordinate.

\begin{lemma}\label{lem:HLY-decomposition}
With the notation above,
$$
  \Lambda(f)
  =
  M
  +
  \E_x f(x)J_{2,1}(u,c)(x)
  +
  \E_x f(x)J_{2,2}(r,v)(x)
  +
  \E_x f(x)J_3(u,v)(x),
$$
where
$$
  M=\E_{x_1,x_2}f(x_1,x_2)c(x_1)r(x_2), 
$$
and
  \begin{align*}
  J_{2,1}(u,c)(x)
  =c(x_1)\E_{y\in\Fp}u(x_1+\phi_1(y),x_2), \qquad
  J_{2,2}(r,v)(x)
  =r(x_2)\E_{y\in\Fp}v(x_1,x_2+\phi_2(y)).
\end{align*}
Moreover,
$$
  \|J_{2,1}(u,c)\|_2\lesssim p^{-1/4}\|u\|_4\|c\|_4,
  \qquad
  \|J_{2,2}(r,v)\|_2\lesssim p^{-1/4}\|r\|_4\|v\|_4.
$$
All implied constants depend only on $\deg\phi_1$ and $\deg\phi_2$.
\end{lemma}

\begin{proof}
Since $f=r+u=c+v$, bilinearity of the averaging operator gives
$A(f,f)=A(r,c)+A(u,c)+A(r,v)+A(u,v)$.
The function $r$ is unchanged by translation in the first coordinate,
while $c$ is unchanged by translation in the second coordinate.  Hence
$A(r,c)(x)=r(x_2)c(x_1)$,
$A(u,c)(x)=J_{2,1}(u,c)(x)$,
$A(r,v)(x)=J_{2,2}(r,v)(x)$, and $A(u,v)=J_3(u,v)$.
Multiplying this identity by $f(x)$ and averaging in $x$ proves the
decomposition.  Applying \cite[Lemma~4.1]{HLY} with
$(f_1,f_2)=(u,c)$ proves the first displayed estimate.  After interchanging
the two coordinates, the same lemma with $(f_1,f_2)=(r,v)$ proves the
second.
\end{proof}

}

\medskip

\section[A refined fibre-energy estimate for $J_3$]
{A refined fibre-energy estimate for
\texorpdfstring{$J_3$}{J3}}
\label{sec:refined-fibre-energy}

The quantitative improvement in the counting argument comes from retaining
a finer quantity in the Han--Lacey--Yang estimate for $J_3$.  In the proof
of \cite[Lemma~4.2]{HLY}, a fibre-square energy arises before it
is bounded by a global $L^4$-norm.  Instead of making that final
majorization, we keep the fibre energy itself.

For a function $g:\Fp^2\to\mathbb C$, define
$$
  \mathcal R(g)
  =
  \E_{x_2\in\Fp}
  \left(
    \E_{x_1\in\Fp}|g(x_1,x_2)|^2
  \right)^2
\qquad\text{and}\qquad
  \mathcal C(g)
  =
  \E_{x_1\in\Fp}
  \left(
    \E_{x_2\in\Fp}|g(x_1,x_2)|^2
  \right)^2.
$$
The quantities $\mathcal R(g)$ and $\mathcal C(g)$ measure,
respectively, the square energies of the $L^2$-masses of $g$ along
horizontal and vertical fibres.

\begin{lemma}\label{lem:refined-HLY}
Assume Kernel Estimate~\ref{est:HLY-kernel}.  Let
$g_1,g_2:\Fp^2\to\mathbb C$ satisfy
$$
  \E_{x_1\in\Fp}g_1(x_1,x_2)=0
  \quad\text{for every }x_2\in\Fp,
$$
and
$$
  \E_{x_2\in\Fp}g_2(x_1,x_2)=0
  \quad\text{for every }x_1\in\Fp.
$$
Then
$$
  \|J_3(g_1,g_2)\|_2^2
  \lesssim
  p^{-1}\|g_1\|_2^2\|g_2\|_2^2
  +
  p^{-1/4}
  \mathcal R(g_1)^{1/2}
  \mathcal C(g_2)^{1/2}.
$$
\end{lemma}

\begin{proof}
We follow the proof of \cite[Lemma~4.2]{HLY}, stopping before the final
$L^4$-majorization.

{By Fourier expansion,}
$$
  J_3(g_1,g_2)(x)
  =
  \sum_{n,m\in\Fp^2}
  \widehat g_1(n-m)
  \widehat g_2(m)
  \widetilde K(n_1-m_1,m_2)
  \psi(n\cdot x).
$$
Here and below, sums over $n$, $m$, and $H$ are taken over
$\Fp^2$, unless otherwise indicated.

Applying Plancherel in $x$, expanding the resulting square, and grouping
pairs of $m$-frequencies according to their difference
{$H=(H_1,H_2)\in\Fp^2$, with the second $m$-frequency
written as $m+H$,}
one obtains the decomposition
$$
  \|J_3(g_1,g_2)\|_2^2
  =
  \sum_{H\in\Fp^2}I(H),
$$
where $I(H)$ denotes the contribution corresponding to the frequency
difference $H$.

The diagonal contribution $H=0$ is estimated exactly as in
\cite[Lemma~4.2]{HLY}:
$$
  I(0)
  \lesssim
  p^{-1}\|g_1\|_2^2\|g_2\|_2^2.
$$

{
For $H\neq0$, Kernel Estimate~\ref{est:HLY-kernel}, together with the
duality argument in Lemma~\ref{lem:HLY-matrix-reduction}, gives
$$
  |I(H)|
  \lesssim
  p^{-1/4}A_H(g_1)B_H(g_2),
$$
where
$$
  A_H(g)
  =
  \left(
    \sum_{n_1\in\Fp}
    \left|
      \sum_{n_2\in\Fp}
      \widehat g(n_1,n_2)
      \overline{
        \widehat g(n_1-H_1,n_2-H_2)
      }
    \right|^2
  \right)^{1/2},
$$
and
$$
  B_H(g)
  =
  \left(
    \sum_{m_2\in\Fp}
    \left|
      \sum_{m_1\in\Fp}
      \widehat g(m_1,m_2)
      \overline{
        \widehat g(m_1+H_1,m_2+H_2)
      }
    \right|^2
  \right)^{1/2}.
$$
Consequently, the triangle inequality and Cauchy--Schwarz in $H$ give
\begin{align*}
  \bigg|\sum_{H\neq0}I(H)\bigg|
  &\leq
  \sum_{H\neq0}|I(H)| 
  \lesssim
  p^{-1/4}
  \left(\sum_H A_H(g_1)^2\right)^{1/2}
  \left(\sum_H B_H(g_2)^2\right)^{1/2}.
\end{align*}
}
It therefore remains to identify the two sums on the right-hand side with
the fibre-square energies.

For $g:\Fp^2\to\mathbb C$, define its normalized partial Fourier
transform in the first coordinate by
$$
  G_{n_1}(x_2)
  =
  \E_{x_1\in\Fp}
  g(x_1,x_2)\psi(-n_1x_1).
$$
Then
$$
  \widehat g(n_1,n_2)
  =
  \E_{x_2\in\Fp}
  G_{n_1}(x_2)\psi(-n_2x_2).
$$

For $H=(H_1,H_2)$, we have
$$
  A_H(g)^2
  =
  \sum_{n_1\in\Fp}
  \left|
    \sum_{n_2\in\Fp}
    \widehat g(n_1,n_2)
    \overline{
      \widehat g(n_1-H_1,n_2-H_2)
    }
  \right|^2.
$$
Expanding the square gives
\begin{align*}
  A_H(g)^2
  =
  \sum_{n_1\in\Fp}
  \sum_{n_2,n_2'\in\Fp}
  &\widehat g(n_1,n_2)
  \overline{
    \widehat g(n_1-H_1,n_2-H_2)
  }  \ 
  \overline{\widehat g(n_1,n_2')}
  \widehat g(n_1-H_1,n_2'-H_2).
\end{align*}
Summing this identity over $H=(H_1,H_2)$ yields
\begin{align*}
  \sum_H A_H(g)^2
  =
  \sum_{H_1,H_2}
  \sum_{n_1}
  \sum_{n_2,n_2'}
  &\widehat g(n_1,n_2)\ 
  \overline{
    \widehat g(n_1-H_1,n_2-H_2)
  } \ 
  \overline{\widehat g(n_1,n_2')}\ 
  \widehat g(n_1-H_1,n_2'-H_2).
\end{align*}
Write $n_1'=n_1-H_1$.  As $H_1$ ranges over $\Fp$, the pair
$(n_1,n_1')$ ranges freely
over $\Fp^2$.  Hence
\begin{align*}
  \sum_H A_H(g)^2
  =
  \sum_{n_1,n_1'}
  \sum_{H_2}
  \sum_{n_2,n_2'}
  &\widehat g(n_1,n_2)\ 
  \overline{
    \widehat g(n_1',n_2-H_2)
  } \ 
  \overline{\widehat g(n_1,n_2')}
 \ \widehat g(n_1',n_2'-H_2).
\end{align*}

For fixed $n_1,n_1'$, the inner expression is the Fourier-side
correlation identity associated with the second coordinate.  By
Plancherel in that coordinate,
\begin{align*}
  &\sum_{H_2}
  \left|
    \sum_{n_2}
    \widehat g(n_1,n_2)
    \overline{
      \widehat g(n_1',n_2-H_2)
    }
  \right|^2 
  =
  \E_{x_2\in\Fp}
  |G_{n_1}(x_2)|^2
  |G_{n_1'}(x_2)|^2.
\end{align*}
Consequently,
\begin{align*}
  \sum_H A_H(g)^2
  &=
  \sum_{n_1,n_1'}
  \E_{x_2}
  |G_{n_1}(x_2)|^2
  |G_{n_1'}(x_2)|^2 
  =
  \E_{x_2}
  \bigg(
    \sum_{n_1}|G_{n_1}(x_2)|^2
  \bigg)^2.
\end{align*}
Finally, Plancherel in the first coordinate gives, for every fixed
$x_2$,
$$
  \sum_{n_1\in\Fp}|G_{n_1}(x_2)|^2
  =
  \E_{x_1\in\Fp}|g(x_1,x_2)|^2.
$$
Therefore,
\begin{align*}
  \sum_H A_H(g)^2
  &=
  \E_{x_2}
  \left(
    \E_{x_1}|g(x_1,x_2)|^2
  \right)^2 
  =
  \mathcal R(g).
\end{align*}

Applying the same calculation with the two coordinates interchanged
{and then replacing $H$ by $-H$} gives
$$
  \sum_H B_H(g)^2
  =
  \E_{x_1}
  \left(
    \E_{x_2}|g(x_1,x_2)|^2
  \right)^2
  =
  \mathcal C(g).
$$
It follows that
$$
 {\left|\sum_{H\neq0}I(H)\right|}
  \lesssim
  p^{-1/4}
  \mathcal R(g_1)^{1/2}
  \mathcal C(g_2)^{1/2}.
$$
Combining this estimate with the diagonal bound for $I(0)$ proves
$$
  \|J_3(g_1,g_2)\|_2^2
  \lesssim
  p^{-1}\|g_1\|_2^2\|g_2\|_2^2
  +
  p^{-1/4}
  \mathcal R(g_1)^{1/2}
  \mathcal C(g_2)^{1/2}.
$$
The proof is complete.
\end{proof}

\begin{remark}
Lemma~\ref{lem:refined-HLY} should not be interpreted as a global
$L^{8/3}\times L^{8/3}\longrightarrow L^2$ bilinear smoothing estimate.
It is the intermediate fibre-energy form of
the Han--Lacey--Yang $J_3$-estimate, obtained before the final
majorization of the fibre energies by global $L^4$-norms.
\end{remark}

\medskip
\section[The strengthened main term and the $p^{-1/14}$ threshold]
{The strengthened main term and the
\texorpdfstring{$p^{-1/14}$}{p-minus-one-fourteenth} threshold}
\label{sec:threshold}

The refined error estimate from
Lemma~\ref{lem:refined-HLY} will be combined with lower bounds for the
main term that are sensitive to the distribution of $A$ along rows and
columns.

Let $A\subset\Fp^2$, let $f=\mathbf 1_A$ and
$\delta=\E_{x\in\Fp^2}f(x)=|A|/p^2$, and define the normalized row and
column densities
$$
  r(x_2)
  =
  \E_{x_1\in\Fp}f(x_1,x_2),
  \qquad
  c(x_1)
  =
  \E_{x_2\in\Fp}f(x_1,x_2).
$$
Thus $r(x_2)$ is the density of $A$ in the row indexed by $x_2$,
while $c(x_1)$ is the density of $A$ in the column indexed by $x_1$.

The main term in the Han--Lacey--Yang decomposition is
$$
  M
  =
  \E_{x_1,x_2}
  f(x_1,x_2)c(x_1)r(x_2).
$$
We shall use two lower bounds for $M$.  The first is the standard
path-counting inequality for a bipartite graph.  The second is sensitive
to the square energies of the row and column degrees and is the key to
the $p^{-1/14}$ absorption argument.

\begin{lemma}\label{lem:path-main}
For $f=\mathbf 1_A$, one has
$$
  M\geq\delta^3.
$$
Equivalently, if
$$
  d(x_1)
  =
  \bigl|\{x_2\in\Fp:(x_1,x_2)\in A\}\bigr|,\qquad
  e(x_2)
  =
  \bigl|\{x_1\in\Fp:(x_1,x_2)\in A\}\bigr|,
$$
and $N=|A|$, then
$$
  \sum_{(x_1,x_2)\in A}d(x_1)e(x_2)
  \geq
  \frac{N^3}{p^2}.
$$
\end{lemma}

\begin{proof}
Write
$$
  S
  =
  \sum_{(x_1,x_2)\in A}d(x_1)e(x_2).
$$
This sum counts homomorphic copies of the three-edge path in the
bipartite graph $A\subset\Fp\times\Fp$, with the middle edge
distinguished.

Since $d(x_1)\geq1$ and $e(x_2)\geq1$ whenever
$(x_1,x_2)\in A$, weighted H\"older gives
\begin{align*}
  N
  &=
  \sum_{(x_1,x_2)\in A}
  \bigl(d(x_1)e(x_2)\bigr)^{1/3}
  \bigl(d(x_1)e(x_2)\bigr)^{-1/3}
  \\
  &\leq
  \bigg(
    \sum_{(x_1,x_2)\in A}d(x_1)e(x_2)
  \bigg)^{1/3}
  \bigg(
    \sum_{(x_1,x_2)\in A}
    \frac{1}{\sqrt{d(x_1)e(x_2)}}
  \bigg)^{2/3}.
\end{align*}
By Cauchy--Schwarz,
\begin{align*}
  \sum_{(x_1,x_2)\in A}
  \frac{1}{\sqrt{d(x_1)e(x_2)}}
  &\leq
  \bigg(
    \sum_{(x_1,x_2)\in A}\frac1{d(x_1)}
  \bigg)^{1/2}
  \bigg(
    \sum_{(x_1,x_2)\in A}\frac1{e(x_2)}
  \bigg)^{1/2}.
\end{align*}
For every $x_1$ with $d(x_1)>0$,
$$
  \sum_{x_2:(x_1,x_2)\in A}\frac1{d(x_1)}=1.
$$
Therefore,
$$
  \sum_{(x_1,x_2)\in A}\frac1{d(x_1)}
  =
  \bigl|\{x_1:d(x_1)>0\}\bigr|
  \leq p.
$$
Similarly,
$$
  \sum_{(x_1,x_2)\in A}\frac1{e(x_2)}
  \leq p.
$$
Hence
$$
  \sum_{(x_1,x_2)\in A}
  \frac{1}{\sqrt{d(x_1)e(x_2)}}
  \leq p.
$$
It follows that
$\displaystyle 
  N\leq S^{1/3}p^{2/3},
$
and therefore
$\displaystyle 
  S\geq\frac{N^3}{p^2}.
$

Finally,
$$
  c(x_1)=\frac{d(x_1)}{p},
  \qquad
  r(x_2)=\frac{e(x_2)}{p},
$$
so
\begin{align*}
  M
  &=
  \frac1{p^2}
  \sum_{x_1,x_2}
  f(x_1,x_2)c(x_1)r(x_2)
=
  p^{-4}
  \sum_{(x_1,x_2)\in A}d(x_1)e(x_2)
  =
  p^{-4}S.
\end{align*}
Since $\delta=N/p^2$, the inequality $S\geq N^3/p^2$ gives
$$
  M\geq p^{-4}\frac{N^3}{p^2}
  =
  \left(\frac{N}{p^2}\right)^3
  =
  \delta^3.
$$
The proof is complete.
\end{proof}

\begin{lemma}\label{lem:bipartite}
Let
$
  R_2=\E_{x_2}r(x_2)^2,
  \ 
  C_2=\E_{x_1}c(x_1)^2.
$
Then
$
  M\geq R_2C_2.
$
\end{lemma}

\begin{proof}
Retain the degree notation
$$
  d(x_1)
  =
  \bigl|\{x_2:(x_1,x_2)\in A\}\bigr|,
  \qquad
  e(x_2)
  =
  \bigl|\{x_1:(x_1,x_2)\in A\}\bigr|.
$$
Let
$$
  S
  =
  \sum_{(x_1,x_2)\in A}d(x_1)e(x_2),
$$
and
$$
  D_2=\sum_{x_1\in\Fp}d(x_1)^2,
  \qquad
  E_2=\sum_{x_2\in\Fp}e(x_2)^2.
$$
Since
$$
  M=p^{-4}S,
  \qquad
  C_2=p^{-3}D_2,
  \qquad
  R_2=p^{-3}E_2,
$$
the desired inequality $M\geq R_2C_2$ is equivalent to
$$
  p^2S\geq D_2E_2.
$$

We first estimate $D_2$.  Since each $d(x_1)$ occurs once for every
edge adjacent to $x_1$,
$$
  D_2
  =
  \sum_{(x_1,x_2)\in A}d(x_1).
$$
Writing
$$
  d(x_1)
  =
  \sqrt{d(x_1)e(x_2)}
  \sqrt{\frac{d(x_1)}{e(x_2)}}
$$
and applying Cauchy--Schwarz gives
$$
  D_2
  \leq
  S^{1/2}
  \bigg(
    \sum_{(x_1,x_2)\in A}
    \frac{d(x_1)}{e(x_2)}
  \bigg)^{1/2}.
$$
For every $x_2$ with $e(x_2)>0$,
$$
  \sum_{x_1:(x_1,x_2)\in A}d(x_1)
  \leq
  p\,e(x_2),
$$
because $d(x_1)\leq p$.  Therefore,
\begin{align*}
  \sum_{(x_1,x_2)\in A}
  \frac{d(x_1)}{e(x_2)}
  &=
  \sum_{x_2:e(x_2)>0}
  \frac1{e(x_2)}
  \sum_{x_1:(x_1,x_2)\in A}d(x_1)
  \leq
  \sum_{x_2:e(x_2)>0}p
  \leq p^2.
\end{align*}
Hence
$$
  D_2\leq pS^{1/2}.
$$
By the symmetric argument,
$$
  E_2\leq pS^{1/2}.
$$
Multiplying the two estimates yields
$$
  D_2E_2\leq p^2S,
$$
which is equivalent to $M\geq R_2C_2$.
\end{proof}

\begin{proposition}\label{prop:threshold}
Assume Kernel Estimate~\ref{est:HLY-kernel}.  Then every
$A\subset\Fp^2$ with density $\delta\gg p^{-1/14}$ contains a nontrivial
polynomial corner.
\end{proposition}

\begin{proof}
Let $f=\mathbf 1_A$, and define
$g_1(x_1,x_2)=f(x_1,x_2)-r(x_2)$ and
$g_2(x_1,x_2)=f(x_1,x_2)-c(x_1)$.
Then $\E_{x_1}g_1(x_1,x_2)=0$ for every $x_2$, and
$\E_{x_2}g_2(x_1,x_2)=0$ for every $x_1$.
Thus $g_1$ and $g_2$ satisfy the cancellation assumptions of
Lemma~\ref{lem:refined-HLY}.

Since $f^2=f$, we have, for every $x_2$,
\begin{align*}
  \E_{x_1}|g_1(x_1,x_2)|^2
  &=
  \E_{x_1}|f(x_1,x_2)-r(x_2)|^2
  =
  r(x_2)-r(x_2)^2
  =
  r(x_2)(1-r(x_2)).
\end{align*}
Similarly,
$\E_{x_2}|g_2(x_1,x_2)|^2=c(x_1)(1-c(x_1))$.
Define
$$
  V_R=\E_{x_2}\bigl(r(x_2)(1-r(x_2))\bigr)^2
  \qquad \text{and}\qquad
  V_C=\E_{x_1}\bigl(c(x_1)(1-c(x_1))\bigr)^2,
$$
and $ V=V_RV_C$.
By construction, $\mathcal R(g_1)=V_R$ and $\mathcal C(g_2)=V_C$.

Moreover,
\begin{align*}
  \|g_1\|_2^2
  &=
  \E_{x_2}r(x_2)(1-r(x_2))
  \leq
  \E_{x_2}r(x_2)
  =
  \delta,
\end{align*}
and similarly $\|g_2\|_2^2\leq\delta$.
Lemma~\ref{lem:refined-HLY} therefore gives
$$
  \|J_3(g_1,g_2)\|_2^2
  \lesssim
  p^{-1}\delta^2
  +
  p^{-1/4}V^{1/2}.
$$
Using $\sqrt{a+b}\leq\sqrt a+\sqrt b$, we obtain
$$
  \|J_3(g_1,g_2)\|_2
  \lesssim
  p^{-1/2}\delta
  +
  p^{-1/8}V^{1/4}.
$$
{
Since $f$, $g_1$, and $g_2$ are real-valued, so is $J_3(g_1,g_2)$.
Cauchy--Schwarz and $\|f\|_2=\delta^{1/2}$ therefore give
\begin{align*}
  \left|\E_x f(x)J_3(g_1,g_2)(x)\right|
  &\leq
  \|f\|_2\|J_3(g_1,g_2)\|_2
  \lesssim
  p^{-1/2}\delta^{3/2}
  +
  p^{-1/8}\delta^{1/2}V^{1/4}.
\end{align*}
}

We now compare these errors with the main term.  By
Lemmas~\ref{lem:path-main} and \ref{lem:bipartite},
$$
  M\geq\delta^3,
  \qquad
  M\geq R_2C_2.
$$
Since $0\leq r,c\leq1$,
$$
  \bigl(r(1-r)\bigr)^2\leq r^2
\qquad\text{and}\qquad
  \bigl(c(1-c)\bigr)^2\leq c^2.
$$
Consequently,
$$
  V_R\leq R_2,
  \qquad
  V_C\leq C_2,
$$
and hence
$$
  V=V_RV_C\leq R_2C_2\leq M.
$$
Thus
$$
  M\geq\max(\delta^3,V).
$$

{
Fix $C_0\geq1$ and assume that $\delta\geq C_0p^{-1/14}$.
The first fully non-degenerate error then satisfies
$$
  p^{-1/2}\delta^{3/2}
  \leq
  C_0^{-3/2}p^{-11/28}\delta^3.
$$

For the fibre-energy error, the two cases $V\leq\delta^3$ and
$V>\delta^3$ are covered by the single inequality
$$
  \delta^{1/2}V^{1/4}
  \leq
  \delta^{-7/4}\max(\delta^3,V).
$$
Since $M\geq\max(\delta^3,V)$, it follows that
$$
  p^{-1/8}\delta^{1/2}V^{1/4}
  \leq
  p^{-1/8}\delta^{-7/4}M
  \leq
  C_0^{-7/4}M.
$$
We first choose $C_0$, depending only on the degrees, so that this last
quantity is a sufficiently small multiple of $M$.

It remains to estimate the two partially degenerate terms in
Lemma~\ref{lem:HLY-decomposition}.  Since $0\leq r,c\leq1$ and
$u=f-r$, $v=f-c$, we have
$$
  \|r\|_4+\|c\|_4+\|u\|_4+\|v\|_4
  \lesssim
  \delta^{1/4}.
$$
The estimates in Lemma~\ref{lem:HLY-decomposition}, followed by
Cauchy--Schwarz, therefore give
$$
  \left|\E_x f(x)J_{2,1}(u,c)(x)\right|
  +
  \left|\E_x f(x)J_{2,2}(r,v)(x)\right|
  \lesssim
  p^{-1/4}\delta.
$$
Under the stated density assumption,
$$
  p^{-1/4}\delta
  \leq
  C_0^{-2}p^{-3/28}\delta^3.
$$
Consequently, after fixing $C_0$ as above and then enlarging the finite
exceptional collection of primes if necessary, the exact decomposition in
Lemma~\ref{lem:HLY-decomposition} gives
$$
  |\Lambda(f)-M|
  \leq
  \frac14 M.
$$
In particular,
$$
  \Lambda(f)\geq\frac34M\geq\frac34\delta^3.
$$
}

It remains to remove the forbidden parameter value $y=0$.  We also show
that the finitely many nonzero common roots of $\phi_1$ and $\phi_2$ may
be discarded.  The contribution of $y=0$ to the normalized counting form
is $p^{-1}\E_x f(x)^3=\delta/p$.  The locus
$Z=\{y\in\Fp^\times:\phi_1(y)=\phi_2(y)=0\}$ has cardinality
$O_{\deg}(1)$, since both polynomials are nonzero and have bounded degree.
The total contribution of these additional completely diagonal parameter
values is therefore $O_{\deg}(\delta/p)$.
{
Under $\delta\geq C_0p^{-1/14}$,
$$
  \frac{\delta}{p}
  \leq
  C_0^{-2}p^{-6/7}\delta^3.
$$
Thus, after a further enlargement of the finite exceptional collection of
primes, the total contribution of these boundedly many parameter values is
at most $M/4$.
}

Strictly speaking, the definition of ``nontrivial'' in
Theorem~\ref{thm:main} only requires $y\neq0$, so the nonzero common
zeros of $\phi_1$ and $\phi_2$ need not be removed.  Removing them
proves the slightly stronger conclusion that the parameter $y$ may be
chosen so that $y\neq0$ and the two polynomial increments do not both
vanish.

After discarding these $O_{\deg}(1)$ parameter values, the remaining
corner count is positive.  Hence $A$ contains a nontrivial polynomial
corner.
\end{proof}

\begin{remark}
The comparison of the error terms with the main term in the proof of
Proposition~\ref{prop:threshold} is the only point at which the improved
density threshold enters.  The two lower bounds for the main term separate
the argument into two regimes.

If the fibre-square energy is small, namely $V\leq\delta^3$, then the
elementary lower bound
$$
  M\geq\delta^3
$$
absorbs the fully non-degenerate error.  If the fibre-square energy is
large, namely $V>\delta^3$, then the energy-sensitive lower bound
$$
  M\geq V
$$
absorbs the factor $V^{1/4}$ in the error term.

Thus the essential algebraic inequality is
$$
  p^{-1/8}\delta^{1/2}V^{1/4}
  \lesssim
  \max(\delta^3,V).
$$
{More precisely, if $\delta\geq C_0p^{-1/14}$, the ratio
of the left-hand side to $M$ is at most $C_0^{-7/4}$.  Fixing $C_0$
sufficiently large gives the required absorption at the endpoint.}
\end{remark}

\medskip

\section{Admissible exceptional sets and the matrix estimate}
\label{sec:admissible}

In the Han--Lacey--Yang argument, the operator estimate for the correlation
matrix is obtained from pointwise bounds outside a generalized diagonal.
In the equal-degree resonant case, it is more natural to work with a
slightly larger exceptional set consisting of boundedly many vertical and
horizontal lines together with boundedly many algebraic curves of bounded
degree.

The purpose of this section is twofold.  First, we show that such an
enlargement of the exceptional set does not affect the required
$\ell^2$-operator estimate.  Second, we formulate the precise exponential
sum estimates that remain to be proved in the resonant case.

Recall that if $C(m,m')$ is a kernel on $\Fp\times\Fp$, then its
associated operator is
$$
  (T_CF)(m)
  =
  \sum_{m'\in\Fp}C(m,m')F(m').
$$
We shall repeatedly use Schur's test in the following elementary form:
if
$$
  \sup_{m\in\Fp}\sum_{m'\in\Fp}|C(m,m')|\leq A
\qquad\text{and}\qquad
  \sup_{m'\in\Fp}\sum_{m\in\Fp}|C(m,m')|\leq B,
$$
then
$$
  \|T_C\|_{\ell^2(\Fp)\to\ell^2(\Fp)}
  \leq (AB)^{1/2}.
$$

\begin{definition}
\label{def:admissible}
A subset $E\subset\Fp\times\Fp$ is called \emph{admissible} if it is
contained in a union of the following form:
\begin{enumerate}[label=\rm(\roman*),leftmargin=2.2em]
\item $O(1)$ vertical or horizontal lines;
\item $O(1)$ algebraic curves of degree $O(1)$, none of which has a
vertical or horizontal irreducible component.
\end{enumerate}
All implicit constants are allowed to depend only on
$\deg\phi_1$ and $\deg\phi_2$.
\end{definition}

\begin{lemma}
\label{lem:admissible-matrix}
Let $C(m,m')$ be a kernel on $\Fp\times\Fp$.  Suppose that there is an
admissible set $E\subset\Fp^2$ such that
$$
  |C(m,m')|
  \lesssim p^{-3/2}
  \qquad
  \bigl((m,m')\notin E\bigr),
$$
and
$$
  |C(m,m')|
  \lesssim p^{-1}
  \qquad
  \bigl((m,m')\in E\bigr).
$$
Then
$$
  \|T_CF\|_{\ell^2(\Fp)}
  \lesssim
  p^{-1/2}\|F\|_{\ell^2(\Fp)}.
$$
Equivalently,
$$
  \|T_C\|_{\ell^2(\Fp)\to\ell^2(\Fp)}
  \lesssim p^{-1/2}.
$$
\end{lemma}

\begin{proof}
Write
$
  C=C_{\mathrm{off}}+C_{\mathrm{bad}},
$
where
$$
  C_{\mathrm{bad}}(m,m')
  =
  C(m,m')\mathbf 1_E(m,m')
\qquad\text{and}\qquad 
  C_{\mathrm{off}}(m,m')
  =
  C(m,m')\mathbf 1_{\Fp^2\setminus E}(m,m').
$$

For the off-exceptional part, every row has absolute sum at most
$$
  \sum_{m'\in\Fp}|C_{\mathrm{off}}(m,m')|
  \lesssim
  p\cdot p^{-3/2}
  =
  p^{-1/2}.
$$
The same estimate holds for every column:
$$
  \sum_{m\in\Fp}|C_{\mathrm{off}}(m,m')|
  \lesssim p^{-1/2}.
$$
Schur's test therefore gives
$$
  \|T_{C_{\mathrm{off}}}\|_{\ell^2\to\ell^2}
  \lesssim p^{-1/2}.
$$

We next consider the part supported on $E$.  Let $\Gamma$ be one of
the algebraic curves occurring in Definition~\ref{def:admissible}, and
assume that $\Gamma$ has no vertical or horizontal component.  Since
$\Gamma$ has bounded degree, B\'ezout's theorem implies that every
vertical line
$
  \{m=m_0\}
$
and every horizontal line
$
  \{m'=m_0'\}
$
meets $\Gamma$ in $O(1)$ points, unless that line is a component of
$\Gamma$.  The latter possibility is excluded by the definition of
admissibility.

Consequently, the portion of $C_{\mathrm{bad}}$ supported on
$\Gamma(\Fp)$ has $O(1)$ nonzero entries in every row and in every
column.  Since each such entry has size $O(p^{-1})$, its row and column
sums are both $O(p^{-1})$.  Schur's test gives
$$
  \|T_{C_{\mathrm{bad}}\mathbf 1_{\Gamma}}\|_{\ell^2\to\ell^2}
  \lesssim p^{-1}.
$$

Now suppose that $E$ contains a vertical line
$$
  L=\{m=m_0\}.
$$
The operator defined by the kernel $C\mathbf 1_L$ is supported at the
single value $m=m_0$.  Its operator norm is the
$\ell^2$-norm of that row:
$$
  \bigg(
    \sum_{m'\in\Fp}|C(m_0,m')|^2
  \bigg)^{1/2}
  \lesssim
  \left(p\cdot p^{-2}\right)^{1/2}
  =
  p^{-1/2}.
$$
A horizontal line is handled in the same way by passing to the adjoint
operator.

Since an admissible set is contained in the union of only $O(1)$ such
curves and lines, the triangle inequality gives
$$
  \|T_{C_{\mathrm{bad}}}\|_{\ell^2\to\ell^2}
  \lesssim p^{-1/2}.
$$
Combining the estimates for $C_{\mathrm{off}}$ and
$C_{\mathrm{bad}}$, we conclude that
$$
  \|T_C\|_{\ell^2\to\ell^2}
  \lesssim p^{-1/2}.
$$
The proof is complete.
\end{proof}

\begin{proposition}
\label{prop:admissible-estimate}
Suppose that, for every nonzero Han--Lacey--Yang frequency shift
$$
  H\in\Fp^2\setminus\{0\},
$$
there exists an admissible set $E_H\subset\Fp^2$ such that the Han--Lacey--Yang
correlation kernel $C_H(m,m')$ satisfies
$$
  |C_H(m,m')|
  \lesssim p^{-3/2}
  \qquad
  \bigl((m,m')\notin E_H\bigr),
$$
and
$$
  |C_H(m,m')|
  \lesssim p^{-1}
  \qquad
  \bigl((m,m')\in E_H\bigr).
$$
Then Kernel Estimate~\ref{est:HLY-kernel} holds.
\end{proposition}

\begin{proof}
By Lemma~\ref{lem:HLY-matrix-reduction}, the estimate required in the
$TT^*$-argument is the operator bound
$$
  \|T_H\|_{\ell^2(\Fp)\to\ell^2(\Fp)}
  \lesssim p^{-1/2},
$$
where $T_H$ is the operator associated with the correlation matrix
$C_H(m,m')$.

For each nonzero $H$, the hypotheses of
Lemma~\ref{lem:admissible-matrix} apply to $C_H$ and $E_H$.
Consequently,
$$
  \|T_H\|_{\ell^2(\Fp)\to\ell^2(\Fp)}
  \lesssim p^{-1/2}.
$$
This proves Kernel Estimate~\ref{est:HLY-kernel}; taking the square root
in the $TT^*$-step gives the
$p^{-1/4}$ bilinear gain required in the Han--Lacey--Yang estimate for $J_3$.
\end{proof}

\subsection{The remaining equal-degree resonant estimate}
\label{subsec:resonant-target}

Proposition~\ref{prop:admissible-estimate} reduces Kernel
Estimate~\ref{est:HLY-kernel} to two pointwise bounds for the correlation
matrix: a square-root cancellation bound outside an admissible set and a
cruder bound on that set.

{
For polynomials of distinct degrees and for two quadratic polynomials, the
required operator estimate follows from \cite[Lemmas~4.3--4.4]{HLY}.  We
shall also use the following equal-degree consequence of the proof in the
appendix of that paper.

\begin{lemma}\label{lem:HLY-equal-degree-nonresonant}
Suppose that $\deg\phi_1=\deg\phi_2=d>2$, and write
$$
  \phi_1(y)=ay^d+\cdots,
  \qquad
  \phi_2(y)=by^d+\cdots.
$$
If $H=(H_1,H_2)\neq0$ satisfies $aH_1\neq bH_2$, then
$$
  \|T_H\|_{\ell^2(\Fp)\to\ell^2(\Fp)}
  \lesssim p^{-1/2}.
$$
\end{lemma}

\begin{proof}
In the equal-degree calculation of \cite[Appendix~A]{HLY}, the four
coefficients in the second row of the top-degree Jacobian, apart from the
common factor $d$ and the displayed signs, are $bm$, $b(m+H_2)-aH_1$,
$bm'$, and $b(m'+H_2)-aH_1$.
When all four geometric variables are nonzero, rank one forces these four
coefficients to be equal.  If some variables vanish, the corresponding
equalities are omitted.  Since $aH_1-bH_2\neq0$, the solutions of these
finitely many systems lie in the generalized diagonal
$D_H\subset\Fp^2$ constructed in that appendix.  Applying
\cite[Theorem~A.1]{HLY} as in the proof of \cite[Lemma~4.4]{HLY} gives
$$
  |C_H(m,m')|
  \lesssim p^{-3/2}
  \qquad ((m,m')\notin D_H),
$$
while the pointwise Weil estimate \cite[(4.2)]{HLY} gives
$$
  |C_H(m,m')|
  \lesssim p^{-1}
  \qquad ((m,m')\in D_H).
$$
Every row and every column of $D_H$ contains $O_d(1)$ points.  Splitting
the correlation matrix into its generalized-diagonal and off-diagonal
parts and applying Schur's test gives the stated operator bound, exactly as
in the proof of \cite[Lemma~4.3]{HLY}.
\end{proof}

It remains to treat the equal-degree resonant shifts.

Assume that $\deg\phi_1=\deg\phi_2=d>2$.
Write
$$
  \phi_1(y)=ay^d+\cdots,
  \qquad
  \phi_2(y)=by^d+\cdots,
  \qquad
  a,b\neq0,
$$
and define $\alpha=b/a$ and $D=\phi_2-\alpha\phi_1$.
The leading terms cancel, so $\deg D<d$.  Linear independence gives
$D\neq0$, while $D(0)=\phi_2(0)-\alpha\phi_1(0)=0$.
Thus $D$ cannot be a nonzero constant, and hence $0<\deg D<d$.

By Lemma~\ref{lem:HLY-equal-degree-nonresonant}, it remains to consider a
nonzero shift $H=(H_1,H_2)$ satisfying $aH_1=bH_2$.  Such a shift is
uniquely of the form $H=(\alpha\lambda,\lambda)$ with
$\lambda\in\Fp^\times$.
We fix $\lambda$ throughout the remainder of this subsection.
}

Define $\Phi(u,v)=\phi_1(u)-\phi_1(v)$ and
$\Xi(u,v)=\phi_2(u)-\phi_2(v)$, and write
$\Omega(v)=-\lambda D(v)$.
{
Indeed, substituting $H=(\alpha\lambda,\lambda)$ into the part of the phase
in \eqref{eq:HLY-pminus3} which is independent of $m,m'$ gives
\begin{align*}
  &H_1\bigl(\phi_1(y_2)-\phi_1(y_4)\bigr)
  +H_2\bigl(\phi_2(y_4)-\phi_2(y_2)\bigr) 
  =-\lambda D(y_2)+\lambda D(y_4)
  =\Omega(y_2)-\Omega(y_4).
\end{align*}
}
For $m,t\in\Fp$, define the twisted fibre trace
$$
  \widetilde A_m(t)
  =
  \sum_{\substack{(u,v)\in\Fp^2\\\Phi(u,v)=t}}
  \psi\!\bigl(m\Xi(u,v)+\Omega(v)\bigr).
$$

The corresponding resonant correlation sum is
\begin{equation}
\label{eq:S-lambda-target}
  S_\lambda(m,m')
  =
  \sum_{t\in\Fp}
  \widetilde A_m(t)
  \overline{\widetilde A_{m'}(t)}.
\end{equation}
Expanding the two fibre traces gives
\begin{align*}
  S_\lambda(m,m')
  =
  \sum_{\substack{y_1,y_2,y_3,y_4\in\Fp\\
  \Phi(y_1,y_2)=\Phi(y_3,y_4)}}
  \psi\Bigl(
    &m\Xi(y_1,y_2)+\Omega(y_2) 
    -m'\Xi(y_3,y_4)-\Omega(y_4)
  \Bigr).
\end{align*}

Introduce the affine hypersurface
$$
  W
  =
  \left\{
    (y_1,y_2,y_3,y_4)\in\mathbb A^4:
    \phi_1(y_1)-\phi_1(y_2)
    =
    \phi_1(y_3)-\phi_1(y_4)
  \right\}.
$$
On $W$, define
$$
  f_0
  =
  \Omega(y_2)-\Omega(y_4)=\lambda (D(y_4)-D(y_2)),
$$
and
$$
  f_1
  =
  \Xi(y_1,y_2)=\phi _2(y_1)-\phi _2(y_2),
  \qquad
  f_2
  =
 {-\Xi(y_3,y_4)=\phi _2(y_4)-\phi _2(y_3).}
$$
Then \eqref{eq:S-lambda-target} becomes
\begin{equation}
\label{eq:S-lambda-threefold}
  S_\lambda(m,m')
  =
  \sum_{w\in W(\Fp)}
  \psi\!\bigl(
    f_0(w)+mf_1(w)+m'f_2(w)
  \bigr).
\end{equation}

Thus $S_\lambda(m,m')$ is a two-parameter family of exponential sums on
the three-dimensional hypersurface $W$.  The map governing the two
Fourier parameters is
$$
  F=(f_1,f_2):W\longrightarrow\mathbb A^2.
$$

It is actually more convenient to work with the augmented map
$$
  \widetilde{F}=(f_0,f_1,f_2):W\longrightarrow\mathbb A^3.
$$

{
If $\psi_\lambda(t)=\psi(\lambda t)$, then
$$
  S_\lambda^\psi(m,m')
  =
  S_1^{\psi_\lambda}\!\left(\frac m\lambda,\frac{m'}\lambda\right).
$$
Since $\psi_\lambda$ is nontrivial and the rescaling of $(m,m')$ is
bijective, bounds for $\lambda=1$ that are uniform over nontrivial additive
characters are uniform in $\lambda$.
}

\smallskip

The following are the two estimates that will be proved in
Section~\ref{sec:replacement}.

\begin{enumerate}[label=\rm(\roman*),leftmargin=2.2em]

\item There exists a closed locus
$$
  \mathcal Z_\lambda\subset\mathbb A^2_{m,m'}
$$
of dimension at most one and degree $O_{\deg}(1)$ such that
\begin{equation}
\label{eq:S-good-needed}
  |S_\lambda(m,m')|
  \lesssim p^{3/2}
  \qquad
  \bigl((m,m')\notin \mathcal Z_\lambda(\Fp)\bigr).
\end{equation}
This is a square-root cancellation estimate for the
generic parameters. 

\smallskip

\item Uniformly for all $m,m'\in\Fp$,
\begin{equation}
\label{eq:S-crude-needed}
  |S_\lambda(m,m')|
  \lesssim p^2.
\end{equation}
This cruder estimate is needed only on the exceptional set
$\mathcal Z_\lambda(\Fp)$.

\end{enumerate}

{
These estimates imply the required Han--Lacey--Yang matrix bound.
For the resonant shift $H=(\alpha\lambda,\lambda)$, write
$C_\lambda=C_{(\alpha\lambda,\lambda)}$.
By \eqref{eq:HLY-pminus3}, direct substitution shows that
\begin{equation}
\label{eq:C-S-relation-needed}
  C_\lambda(m,m')
  =
  \mathbf 1_{\substack{
    m\ne0,\ m+\lambda\ne0\\
    m'\ne0,\ m'+\lambda\ne0
  }}
  p^{-3}S_\lambda(m,m').
\end{equation}
In particular, the truncated correlation entry vanishes on the four lines
$m=0$, $m=-\lambda$, $m'=0$, and $m'=-\lambda$.

Define the exceptional locus
$$
  E_\lambda
  =
  \mathcal Z_\lambda(\Fp)
  \cup\{m=0\}
  \cup\{m=-\lambda\}
  \cup\{m'=0\}
  \cup\{m'=-\lambda\}.
$$

Combining \eqref{eq:S-good-needed} with
\eqref{eq:C-S-relation-needed} gives
$$
  |C_\lambda(m,m')|
  \lesssim p^{-3/2}
  \qquad
  \bigl((m,m')\notin E_\lambda\bigr).
$$
The estimate \eqref{eq:S-crude-needed}, together with the vanishing on the
four cutoff lines, gives
$$
  |C_\lambda(m,m')|
  \lesssim p^{-1}
  \qquad
  \bigl((m,m')\in E_\lambda\bigr).
$$

It remains only to verify that $E_\lambda$ is admissible.  Since
$\mathcal Z_\lambda$ has dimension at most one and bounded
degree, it has only $O_{\deg}(1)$ irreducible components.  Each
one-dimensional component is one of the following:

\begin{enumerate}[label=\rm(\alph*),leftmargin=2.2em]
\item a vertical line;
\item a horizontal line;
\item a bounded-degree curve with no vertical or horizontal component.
\end{enumerate}

The zero-dimensional part contains only $O_{\deg}(1)$ geometric points,
by the degree bound.  Its $\Fp$-points may therefore be covered by
$O_{\deg}(1)$ vertical lines.  Adding the four cutoff lines therefore
shows that $E_\lambda$ is admissible.

Once \eqref{eq:S-good-needed} and \eqref{eq:S-crude-needed} have been
established, Lemma~\ref{lem:admissible-matrix}, applied to
$C_\lambda$ and $E_\lambda$ for each shift
$H=(\alpha\lambda,\lambda)$, gives the required operator estimate in the
equal-degree resonant case.  For distinct degrees and for the equal-degree
quadratic case, the estimate follows from
\cite[Lemmas~4.3 and~4.4]{HLY}; for equal degree $d>2$ and a nonresonant
shift, it follows from Lemma~\ref{lem:HLY-equal-degree-nonresonant}.
These cases exhaust all nonzero Han--Lacey--Yang shifts: the equal-degree linear case
cannot occur, because two linear polynomials with zero constant terms are
linearly dependent.  Hence Kernel Estimate~\ref{est:HLY-kernel} follows
once the two resonant exponential-sum estimates have been proved.
}

\medskip
\medskip

\section{The proofs of the remaining estimates concerning exponential sums}
\label{sec:replacement}

In this section we provide the proofs for the needed exponential sums estimates, and hence finish the proof of the main theorem.

\subsection{The resonant exponential sum problem}
\label{subsec:what-we-need}

We recall the resonant setup from Section~5 and formulate the two
exponential sum estimates that will be proved in this section.

Assume that $\phi _1,\phi _2\in \mathbb{Z}[y]$ with $\phi _1(0)=\phi _2(0)=0$,
$$
  \deg\phi_1=\deg\phi_2=d>2,
$$
and are not linearly dependent.  Write 
$$
  \phi_1(y)=ay^d+\cdots,
  \qquad
  \phi_2(y)=by^d+\cdots,
  \qquad
  a,b\neq0,
$$
and define
$$
  \alpha=\frac{b}{a}.
$$
In the resonant Han--Lacey--Yang kernel, the two relevant frequency parameters are of
the form
$$
  \eta_1=\alpha\lambda,
  \qquad
  \eta_2=\lambda,
  \qquad
  \lambda\in\mathbb F_p^\times.
$$
For the analysis of a single correlation sum, we fix
$\lambda\in\mathbb F_p^\times$.  The estimates obtained below will be
uniform in the choice of $\lambda$.

Write
$$
  D=\phi_2-\alpha\phi_1.
$$
Since the leading terms cancel, we have $\deg D<d$.  Moreover,
$D$ is not a nonzero constant because $\phi_1$ and $\phi_2$ are linearly independent and $\phi _1(0)=\phi _2(0)=0$.
Hence 
$$
  0<\deg D<d.
$$

Define
$$
  W
  =
  \left\{
    (y_1,y_2,y_3,y_4)\in\mathbb A^4:
    \phi_1(y_1)-\phi_1(y_2)
    =
    \phi_1(y_3)-\phi_1(y_4)
  \right\},
$$
and
$$
\begin{array}{rcl}
  f_0(y_1,y_2,y_3,y_4)
  &=&\lambda (D(y_4)-D(y_2)),\\
  f_1(y_1,y_2,y_3,y_4)
  &=&\phi _2(y_1)-\phi _2(y_2),\\
  f_2(y_1,y_2,y_3,y_4)
  &=&\phi _2(y_4)-\phi _2(y_3).
\end{array}
$$
We also define the exponential sum on $W$:
$$
  S_\lambda(m,m')
  =
  \sum_{w\in W(\mathbb F_p)}
  \psi\bigl(f_0(w)+mf_1(w)+m'f_2(w)\bigr).
$$

{The two estimates needed to complete the resonant case can
now be stated as follows; they will be proved in the remainder of this
section.}

\begin{lemma}
\label{lem:stratified-sqrt-S1}
Given a positive integer $d$. There exist $p_0>0$ and a closed subset
$$
  \mathcal Z_\lambda\subset\mathbb A^2_{m,m'}
$$
of dimension at most one and degree $O_{\deg}(1)$ such that for all $p>p_0$, and all pairs of polynomials $\phi _1,\phi _2\in \mathbb{Z}[y]$ of degree $\leq d$, vanishing at $y=0$ and are linearly independent after base change to $\mathbb{F}_p$, then
$$
  S_\lambda(m,m')
  =
  O_{\deg}(p^{3/2})
  \qquad
  \bigl((m,m')\notin\mathcal Z_\lambda(\mathbb F_p)\bigr).
$$
Moreover, uniformly for all
$$
  (m,m')\in\mathbb A^2(\mathbb F_p),
$$
one has
$$
  S_\lambda(m,m')
  =
  O_{\deg}(p^2).
$$
The implied constants and the degree of $\mathcal Z_\lambda$ are independent of $p$ and $\lambda\in\mathbb F_p^\times.$
\end{lemma}

\begin{remark} For simplicity, in what follows we fix
$\lambda=1$, and hence $f_0=D(y_4)-D(y_2)$.  The proof for other nonzero
values of $\lambda$ is similar.
\end{remark}

\subsection{Fibres and images of the relevant maps}
\label{subsec:fibres-images}

We shall use the following two properties of the fibres and images of these
maps. {Here we work over $\mathbb{Z}$, and in the next subsection will base change to $\mathbb{F}_p$ for $p$ larger than an appropriate number $N$.}

\medskip

\noindent
\textbf{Property 1.}
The augmented map has finite generic fibres.

\medskip

\noindent
\textbf{Property 2.}
The augmented map contracts no surface to a point.

\medskip

We recall the augmented map from the previous subsection 
$$
  \widetilde{F}=(f_0,f_1,f_2):W\longrightarrow\mathbb A^3.
$$

\begin{lemma} 
\label{lem:augmented-map-fibres}

Every geometric fibre of $\widetilde{F}$ has dimension at most one.
\end{lemma}

\begin{proof}
Fix
$$
  (u,v,c)\in\mathbb A^3 {(\mathbb{C})}.
$$
Before imposing the equation defining $W$, the equations
$$
  \phi _2(y_1)-\phi _2(y_2)=u \qquad{and}\qquad
  -\phi _2(y_3)+\phi _2(y_4)=v
$$
define a product
$$
  C_u\times C_v'
$$
of two affine plane curves.

Every irreducible component of $C_u$ dominates both coordinate lines.
Indeed, if $y_1=c_0$ were a vertical component, then the polynomial
$$
  \phi _2(c_0)-\phi _2(y_2)-u
$$
would vanish identically in $y_2$, which is impossible because $\phi _2$ is
nonconstant.  The same argument applies to the other coordinate projection,
and likewise to the components of $C_v'$.

Let $C\times C'$ be an irreducible product component of
$C_u\times C_v'$.  The function
$$
  f_0=D(y_4)-D(y_2)
$$
is nonconstant on $C\times C'$.  Indeed, if $G$ were constant, then,
after fixing a general point of $C$, the function $D(y_4)$ would be
constant on $C'$.  Since the projection
$$
  C'\longrightarrow\mathbb A^1_{y_4}
$$
is dominant, this would force $D$ itself to be constant, contradicting
$$
  \deg D\geq1.
$$

It follows that the equation
$$
  f_0=c
$$
cuts every two-dimensional product component properly.  Hence the resulting
intersection has dimension at most one.  Imposing in addition the equation
defining $W$ cannot increase the dimension.  Therefore every geometric
fibre of $\Theta$ has dimension at most one.
\end{proof}

In particular, no surface contained in $W$ can be mapped by $\widetilde{F}$ to
a point.  Thus Lemma~\ref{lem:augmented-map-fibres} establishes
Property~2.

\begin{lemma} 
\label{lem:augmented-map-generic-finiteness}
After excluding finitely many characteristics depending only on $d$, the
restriction of $\Theta$ to every three-dimensional {geometrically} irreducible component
of $W$ is generically finite.
\end{lemma}

\begin{proof}
We recall that $W$ is the hypersurface $H=0$, where
$$
  H
  =
  \phi _1(y_1)-\phi _1(y_2)-\phi _1(y_3)+\phi _1(y_4).
$$

Consider the map $\Psi: \mathbb{A}^4\rightarrow \mathbb{A}^4$ defined by
$$
  \Psi
  =
  \Bigl(
    D(y_4)-D(y_2),
    \phi _2(y_1)-\phi _2(y_2),
    -\phi _2(y_3)+\phi _2(y_4),
    D(y_1)-D(y_3)
  \Bigr).
$$
Write its four components as
$$
  A=\phi _2(y_1)-\phi _2(y_2),\qquad
  B=-\phi _2(y_3)+\phi _2(y_4),
$$
$$
  G=D(y_4)-D(y_2),\qquad
  K=D(y_1)-D(y_3).
$$
Since
$$
  \phi _2=\alpha \phi _1+D,
$$
 we have
$$
  H=\alpha^{-1}(A+B-G-K).
$$
Consequently, at a {\bf smooth} {geometric} point on $W$, the corresponding Jacobian determinants of $\widetilde{F}$ and $\Psi$ differ only by a non-zero factor. {Indeed,} a point $x=(y_1,y_2,y_3,y_4)\in W$ is singular iff $\phi _1'(y_1)=\phi _1'(y_2)=\phi _1'(y_3)=\phi _1'(y_4)=0$. Therefore, if $x^{(0)}=(y_1^{(0)},y_2^{(0)},y_3^{(0)},y_4^{(0)})\in W$ is smooth, then without loss of generality we can assume that $\frac{\partial H}{\partial y_4}(x^{(0)})=\phi _1'(y_4^{(0)})\not= 0$, and near $x^{(0)}$ the coordinates $y_1,y_2,y_3$ are local coordinates on $W$. Since $\widetilde{F}=(G,A,B)$, the Jacobian matrix of $\widetilde{F}$ at the point $x^{(0)}${, with respect to the coordinates $y_1,y_2,y_3$, is}
 $$
	J \widetilde{F}= 
	\left[ \begin{array}{ccc}
		\frac{\partial G}{\partial y_1}+ \frac{\partial G}{\partial y_4}\frac{dy_4}{dy_1}&\frac{\partial G}{\partial y_2}+ \frac{\partial G}{\partial y_4}\frac{dy_4}{dy_2}&\frac{\partial G}{\partial y_3}+ \frac{\partial G}{\partial y_4}\frac{dy_4}{dy_3}\\[8pt]
		\frac{\partial A}{\partial y_1}+ \frac{\partial A}{\partial y_4}\frac{dy_4}{dy_1}&\frac{\partial A}{\partial y_2}+ \frac{\partial A}{\partial y_4}\frac{dy_4}{dy_2}&\frac{\partial A}{\partial y_3}+ \frac{\partial A}{\partial y_4}\frac{dy_4}{dy_3}\\[8pt]
        \frac{\partial B}{\partial y_1}+ \frac{\partial B}{\partial y_4}\frac{dy_4}{dy_1}&\frac{\partial B}{\partial y_2}+ \frac{\partial B}{\partial y_4}\frac{dy_4}{dy_2}&\frac{\partial B}{\partial y_3}+ \frac{\partial B}{\partial y_4}\frac{dy_4}{dy_3}\\
	\end{array} \right].
	$$

By implicit differentiation, at $x^{(0)}$ we have 
$$\frac{dy_4}{dy_1}=-\frac{\frac{\partial H}{\partial y_1}}{\frac{\partial H}{\partial y_4}},\ \ \frac{dy_4}{dy_2}=-\frac{\frac{\partial H}{\partial y_2}}{\frac{\partial H}{\partial y_4}},\ \  \frac{dy_4}{dy_3}=-\frac{\frac{\partial H}{\partial y_3}}{\frac{\partial H}{\partial y_4}}.$$

Therefore, at $x^{(0)}$: 
$$
	J \widetilde{F}= 
	\left[ \begin{array}{ccc}
		\frac{\partial G}{\partial y_1}- \frac{\partial G}{\partial y_4}\frac{\frac{\partial H}{\partial y_1}}{\frac{\partial H}{\partial y_4}}&\frac{\partial G}{\partial y_2}- \frac{\partial G}{\partial y_4}\frac{\frac{\partial H}{\partial y_2}}{\frac{\partial H}{\partial y_4}}&\frac{\partial G}{\partial y_3}- \frac{\partial G}{\partial y_4}\frac{\frac{\partial H}{\partial y_3}}{\frac{\partial H}{\partial y_4}}\\[12pt]
		\frac{\partial A}{\partial y_1}- \frac{\partial A}{\partial y_4}\frac{\frac{\partial H}{\partial y_1}}{\frac{\partial H}{\partial y_4}}&\frac{\partial A}{\partial y_2}- \frac{\partial A}{\partial y_4}\frac{\frac{\partial H}{\partial y_2}}{\frac{\partial A}{\partial y_4}}&\frac{\partial A}{\partial y_3}- \frac{\partial A}{\partial y_4}\frac{\frac{\partial H}{\partial y_3}}{\frac{\partial H}{\partial y_4}}\\[12pt]
        \frac{\partial B}{\partial y_1}- \frac{\partial B}{\partial y_4}\frac{\frac{\partial H}{\partial y_1}}{\frac{\partial H}{\partial y_4}}&\frac{\partial B}{\partial y_2}- \frac{\partial B}{\partial y_4}\frac{\frac{\partial H}{\partial y_2}}{\frac{\partial B}{\partial y_4}}&\frac{\partial B}{\partial y_3}- \frac{\partial B}{\partial y_4}\frac{\frac{\partial H}{\partial y_3}}{\frac{\partial H}{\partial y_4}}\\
	\end{array} \right].
	$$

Since $\Psi =(G,A,B,K)$ and the relation $\alpha H=A+B-G-K$, where $\alpha \not= 0$, at $x^{(0)}$ the Jacobian determinant of $\Psi$ is a non-zero multiple of that for the matrix: 
 $$
	J_1= 
	\left[ \begin{array}{cccc}
		\frac{\partial G}{\partial y_1}&\frac{\partial G}{\partial y_2}&\frac{\partial G}{\partial y_3}&\frac{\partial G}{\partial y_4}\\[8pt]
		\frac{\partial A}{\partial y_1}&\frac{\partial A}{\partial y_2}&\frac{\partial A}{\partial y_3}&\frac{\partial A}{\partial y_4}\\[8pt]
        \frac{\partial B}{\partial y_1}&\frac{\partial B}{\partial y_2}&\frac{\partial B}{\partial y_3}&\frac{\partial B}{\partial y_4}\\[8pt]
        \frac{\partial H}{\partial y_1}&\frac{\partial H}{\partial y_2}&\frac{\partial H}{\partial y_3}&\frac{\partial H}{\partial y_4}\\
	\end{array} \right].
	$$

$J_1$ has the same determinant as the following matrix $J_2$, which is obtained by the following column operation: for $k=1,2,3$, the $k$-th column $=$ the $k$-th column minus $\frac{\frac{\partial H}{\partial y_k}}{\frac{\partial H}{\partial y_4}}$ {times the $4$-th column}: 
$$
	J_2= 
	\left[ \begin{array}{cccc}
		\frac{\partial G}{\partial y_1}- \frac{\partial G}{\partial y_4}\frac{\frac{\partial H}{\partial y_1}}{\frac{\partial H}{\partial y_4}}&\frac{\partial G}{\partial y_2}- \frac{\partial G}{\partial y_4}\frac{\frac{\partial H}{\partial y_2}}{\frac{\partial H}{\partial y_4}}&\frac{\partial G}{\partial y_3}- \frac{\partial G}{\partial y_4}\frac{\frac{\partial H}{\partial y_3}}{\frac{\partial H}{\partial y_4}}&\frac{\partial G}{\partial y_4}\\[10pt]
		\frac{\partial A}{\partial y_1}- \frac{\partial A}{\partial y_4}\frac{\frac{\partial H}{\partial y_1}}{\frac{\partial H}{\partial y_4}}&\frac{\partial A}{\partial y_2}- \frac{\partial A}{\partial y_4}\frac{\frac{\partial H}{\partial y_2}}{\frac{\partial A}{\partial y_4}}&\frac{\partial A}{\partial y_3}- \frac{\partial A}{\partial y_4}\frac{\frac{\partial H}{\partial y_3}}{\frac{\partial H}{\partial y_4}}&\frac{\partial A}{\partial y_4}\\[10pt]
        \frac{\partial B}{\partial y_1}- \frac{\partial B}{\partial y_4}\frac{\frac{\partial H}{\partial y_1}}{\frac{\partial H}{\partial y_4}}&\frac{\partial B}{\partial y_2}- \frac{\partial B}{\partial y_4}\frac{\frac{\partial H}{\partial y_2}}{\frac{\partial B}{\partial y_4}}&\frac{\partial B}{\partial y_3}- \frac{\partial B}{\partial y_4}\frac{\frac{\partial H}{\partial y_3}}{\frac{\partial H}{\partial y_4}}&\frac{\partial B}{\partial y_4}\\[10pt]
        0&0&0&\frac{\partial H}{\partial y_4}
	\end{array} \right].
	$$
Hence, at $x^{(0)}$ we have $\det (J_1)=\det (J_2)=-\frac{\partial H}{\partial y_4}\det (J \widetilde{F})$, which is a non-zero multiple of $\det (J \widetilde{F})$, as wanted.

A direct computation gives
$$
  J_\Psi
  =
  \phi _2'(y_1)\phi _2'(y_4)D'(y_2)D'(y_3)
  -
  \phi _2'(y_2)\phi _2'(y_3)D'(y_1)D'(y_4).
$$
Since the singular part of $W$  is of zero dimension, for the proof of the lemma it is therefore enough to show that no three-dimensional irreducible component of $H=0$ is contained in $J_\Psi=0$. Let
$$
  \deg \phi _2=d,
  \qquad
  \deg D=e,
  \qquad
  1\leq e<d.
$$
The leading homogeneous parts of $H$ and $J_\Psi$ are, up to nonzero
scalar factors,
$$
  H_{\mathrm{top}}
  =
  y_1^d-y_2^d-y_3^d+y_4^d
$$
and
$$
  (J_\Psi)_{\mathrm{top}}
  =
  (y_1y_2y_3y_4)^{e-1}
  \left(
    y_1^{d-e}y_4^{d-e}
    -
    y_2^{d-e}y_3^{d-e}
  \right).
$$
These equations describe the intersections of the projective closures of
$H=0$ and $J_\Psi=0$, respectively, with the hyperplane at infinity.

Suppose that a three-dimensional irreducible component of $H=0$ were
contained in $J_\Psi=0$.  Then a two-dimensional component of its
projective boundary would be contained in the common zero locus of
$$
  H_{\mathrm{top}}
  \quad\text{and}\quad
  (J_\Psi)_{\mathrm{top}}.
$$

No surface component of $H_{\mathrm{top}}=0$ lies in a coordinate
hyperplane, since $H_{\mathrm{top}}$ is not divisible by any of the
coordinates.  We may therefore work in the affine chart of the hyperplane
at infinity given by
$$
  y_1=1.
$$
In this chart, the two equations become
$$
  1-y_2^d-y_3^d+y_4^d=0
$$
and
$$
  y_4^{d-e}=(y_2y_3)^{d-e}.
$$
After excluding characteristics dividing $d-e$, the irreducible
components of the second equation are of the form
$$
  y_4=\zeta y_2y_3,
$$
where $\zeta$ is a $(d-e)$-th root of unity.

If $H_{\mathrm{top}}=0$ and $(J_\Psi)_{\mathrm{top}}=0$ had a common
surface component, then for some such $\zeta$, substituting
$$
  y_4=\zeta y_2y_3
$$
into $H_{\mathrm{top}}$ would make the polynomial
$$
  1-y_2^d-y_3^d+\zeta^d y_2^d y_3^d
$$
vanish identically.  This is impossible.  Hence
$H_{\mathrm{top}}$ and $(J_\Psi)_{\mathrm{top}}$ have no common surface
component.

It follows that $J_\Psi$ does not vanish identically on any
three-dimensional irreducible component of $W$.  Thus $\widetilde{F}$ has
generic rank three on every such component and is therefore generically
finite there.
\end{proof}

Lemma~\ref{lem:augmented-map-generic-finiteness} establishes Property~1.

\subsection{The proof of exponential sum bounds}
\label{subsec:fixed-pair-sheaf}


Let $W$ be a variety of pure dimension $3$ over $\mathbb Z$.
Even if $W$ is irreducible over $\mathbb Z$, its geometric generic
fibre need not be geometrically irreducible.  After base change to an
algebraic closure, it may decompose into finitely many irreducible
components.  In the equidimensional situation considered here, all
relevant {top-dimensional} components have dimension $3$.

Let
$$
  \widetilde F=(f_0,f_1,f_2):W\longrightarrow\mathbb A^3
$$
be a morphism.  For
$$
  m_0,m,m'\in\mathbb F_p,
$$
consider the exponential sum
$$
  S_{m_0,m,m',\widetilde F,W}
  =
  \sum_{x\in W(\mathbb F_p)}
  \psi_p\bigl(
    m_0f_0(x)+mf_1(x)+m'f_2(x)
  \bigr).
$$

For the particular variety $W$ and morphism $\widetilde F$ constructed
above, the two estimates required in Section~5 are the following.  We
emphasize that these statements concern the specific $W$ and
$\widetilde F$ arising in our problem; they are not valid for an
arbitrary variety and morphism.

\medskip

\noindent
\textbf{Generic bound.}
For $m_0=1$ and for $(m,m')\in \mathbb{A}^2(\mathbb{F}_p)$ outside a bounded-degree curve {(descended from a universal curve defined on $\mathbb{Z}$)},
$$
  \left|
    S_{1,m,m',\widetilde F,W}
  \right|
  =
  O(p^{3/2}).
$$

\medskip

\noindent
\textbf{Crude bound.}
For $m_0=1$ and all $(m,m')$,
$$
  \left|
    S_{1,m,m',\widetilde F,W}
  \right|
  =
  O(p^2).
$$

\medskip

We fix
$$
  \psi_p(t)=\exp(2\pi i t/p).
$$
For a tuple $(m_0,m,m'){\in \mathbb{A}^3(\mathbb{Z})}$, define the linear function
$$
  h_{m_0,m,m'}:\mathbb A^3\longrightarrow\mathbb A^1
$$
by
$$
  h_{m_0,m,m'}(z_0,z_1,z_2)
  =
  m_0z_0+mz_1+m'z_2.
$$
When no confusion can arise, we write simply $h$.

{Now we base change to $\mathbb{F}_p$ and} let
$$
  \mathcal L_\psi(h)
$$
denote the Artin--Schreier sheaf associated with the additive character
$$
  z\longmapsto\psi_p(h(z)).
$$
The exponential sum
$$
  S_{m_0,m,m',\widetilde F,W}
$$
corresponds to the pullback sheaf
$$
  \mathcal L_{\psi,h}
  :=
  \widetilde F^*\mathcal L_\psi(h).
$$

Let $H_c^j$ denote $\ell$-adic cohomology with compact support after
base change to $\overline{\mathbb F}_p$. The action of the Frobenius morphism at a geometric stalk of $\mathcal L_{\psi,h}$ is multiplication by $\psi _p(h\circ \widetilde{F}(.))$, hence the Grothendieck--Lefschetz
trace formula gives
$$
  S_{m_0,m,m',\widetilde F,W}
  =
  \sum_{j\geq0}
  (-1)^j
  \operatorname{Tr}
  \left(
    \operatorname{Frob}_p
    \mid
    H_c^j
    \bigl(
      W_{\overline{\mathbb F}_p},
      \mathcal L_{\psi,h}
    \bigr)
  \right).
$$
Here $\operatorname{Frob}_p$ denotes the geometric Frobenius acting on
the cohomology groups.  The fixed points of the corresponding arithmetic
Frobenius are precisely the points of $W(\mathbb F_p)$.

Deligne's Riemann hypothesis over finite fields controls the weights of the
Frobenius eigenvalues.  More precisely, since the above sheaf $\mathcal L_{\psi,h}$ has pointwise weight $0$, an eigenvalue occurring in 
$H_c^j$ has complex absolute value at most $p^{j/2}$.

{For further details on the use of Artin--Schreier
sheaves, the Grothendieck--Lefschetz trace formula, and Deligne's Riemann
hypothesis in estimates for exponential sums, see Expos\'e~VI of
\cite{SGA45}, in particular formula~(1.5.1), Scholie~1.9,
Theorem~1.17, and formula~(1.17.1).}

To derive an exponential sum estimate from the trace formula, we therefore
need two types of information:

\begin{enumerate}
\item[(i)] an upper bound $k$ for the degrees contributing to the trace,
or, in the simplest vanishing formulation,
$$
  H_c^j
  \bigl(
    W_{\overline{\mathbb F}_p},
    \mathcal L_{\psi,h}
  \bigr)
  =
  0
  \qquad (j>k);
$$

\item[(ii)] a uniform upper bound for the Betti numbers
$$
  \dim
  H_c^j
  \bigl(
    W_{\overline{\mathbb F}_p},
    \mathcal L_{\psi,h}
  \bigr).
$$
\end{enumerate}

The relevant constants must be independent of $p$.  Under such bounds,
one obtains an estimate of the form
$$
  \left|
    S_{m_0,m,m',\widetilde F,W}
  \right|
  \leq
  Cp^{k/2}
  +
  O\bigl(p^{(k-1)/2}\bigr).
$$

\begin{remark}
It is not strictly necessary that
$$
  H_c^j=0
  \qquad (j>k).
$$
It is enough that the trace of Frobenius on every higher cohomology group
be $O(p^{k/2})$.  For instance, cancellation may occur among eigenvalues
of larger individual absolute value.
\end{remark}

For the generic bound, we use semiperverse or perverse-sheaf results to
obtain the required support and weight profile.  For the crude bound, we
use the specific geometry of $W$, together with the classical
Bombieri-Weil estimates for exponential sums on algebraic curves.  The
necessary background will be recalled in the appendix.  Both arguments use
the geometric properties of $W$ and $\widetilde F$ established in
Subsection~\ref{subsec:fibres-images}.

\subsubsection{The generic bound}
\label{subsubsec:generic-bound}

We first (in Theorem \ref{thm:generic-fixed-map}) prove the result in case $\phi _1$ and $\phi _2$ are fixed, with implied constants depend only on $\phi _1, \phi _2\in \mathbb{Z}[y]$ and not on the prime number $p$. While this result is weaker than what we need (the implied constants depend only on the degree $d$), it illustrates the main ingredients needed to achieve the needed estimates. After that, we will show that the implied constants depend only on the degree $d$, by using the results and ideas from \cite{BonolisKowalskiWoo}, and hence conclude the proof for the generic bound. 

By the results of Subsection~\ref{subsec:fibres-images}, the variety $W$
and the morphism $\widetilde F$ satisfy the hypotheses of the following
result.  It therefore yields the required generic bound for a fixed pair of $\phi _1, \phi _2\in \mathbb{Z}[y]$, except that the implied constants may depend on $\phi _1,\phi _2$, and not only on the degree $d$. 

\begin{theorem}
\label{thm:generic-fixed-map}
Assume that
$$
  \widetilde F:W\longrightarrow\mathbb A^3
$$
is generically finite on every {geometric} three-dimensional component of $W$.
Assume also that no surface component of $W$ is mapped by
$\widetilde F$ to a zero-dimensional subvariety of $\mathbb A^3$.
Then, for all sufficiently large primes $p$, there exists an exceptional
curve in the plane $m_0=1$ {(descended from the same curve over $\mathbb{Z}$)} such that
$$
  \left|
    S_{1,m,m',\widetilde F,W}
  \right|
  =
  O(p^{3/2})
$$
for every $(m,m')$ outside this curve. {The constants in the above bounds depend only on $W$ and $\widetilde F$, and not on $p$.}
\end{theorem}

\begin{proof}
We describe the stratification argument, emphasizing the points at which
genericity in $(m,m')$ is required.

There exists a hypersurface
$$
  V_1\subset V_0:=\mathbb A^3
$$
such that, with
$$
  W_1=\widetilde F^{-1}(V_1),
$$
the restriction
$$
  \widetilde F:
  W\setminus W_1
  \longrightarrow
  \mathbb A^3\setminus V_1
$$
is smooth. {We note that $V_1$ can be determined from a Jacobian criterion, and hence has degree bounded in terms of $d$. In particular, the number of connected components of $V_1$ is also bounded in terms of $d$. This applies to all varieties and morphisms we will encounter later.}  Since $\widetilde F$ is generically finite, the relative
dimension of this restriction is
$$
  r=0.
$$

For a generic choice of
$$
  (m_0,m,m')=(1,m,m'),
$$
the Katz--Laumon theory gives the required vanishing and support bounds for
$$
  \mathcal L_\psi(h)
$$
on $\mathbb A^3\setminus V_1$, {when the prime number $p$ is large enough}.  We recall, see Section 7.2 in the Appendix, that the
exceptional set in the Katz--Laumon theorem may be taken to be a
homogeneous hypersurface, {defined over $\mathbb{Z}$}, in the variables $m_0,m,m'$.  Consequently, we
may restrict to the affine plane $m_0=1$, and the exceptional locus there
is an algebraic curve.

{Here is a useful remark before we continue the proof. When the prime $p$ is large enough, then the base change of $\widetilde F:
  W\setminus W_1
  \longrightarrow
  \mathbb A^3\setminus V_1$ is also smooth. The same applies for all varieties and morphisms we will encounter later.}

{Pulling back the Artin--Schreier sheaf along the smooth
morphism of relative dimension zero preserves}
the relevant cohomological bound (see Sections 7.1 and 7.2 in the Appendix).  Consequently,
$$
  \left|
    S_{1,m,m',\widetilde F,W\setminus W_1}
  \right|
  =
  O(p^{3/2})
$$
for generic $(m,m')$. {Here we provide details on why this bound is uniform in $p$. As explained in the beginning of Section \ref{subsec:fixed-pair-sheaf}, we only need to bound the Betti number of $\widetilde{F}_{W\setminus W_1}^*(\mathcal L_\psi(h))$.   The results in \cite{KatzLaumon} show that the Betti  numbers of the Fourier transform of  $\widetilde{F}_{W\setminus W_1,!}(\overline{\mathbb{Q}_l})$ are bounded from above by a constant depending only on $W$ and $\widetilde{F}$. (Strictly speaking, we cannot directly apply results in \cite{KatzLaumon} since $\A^3\setminus V_1$ is not closed in $\A^3$, and also the sheaf $\widetilde{F}_{W\setminus W_1}$ may be only semiperverse and not perverse. However, Fourier transform preserves also semiperversity, so this difficulty is resolved. Concerning that $A^3\setminus V_1$ is not closed in $\A^3$, by the exicision property for etale cohomology of l-adic sheaves for the pair $\A^3, V_1$, see Sections 1.4 and 2.2 in \cite{BBD}, we obtain the desired property for $\A^3\backslash V_1$ by using that on $\A^3$ and on $V_1$. For further reference on this, see \cite{FouvryKatz}.)  Now, by Leray spectral sequence, see Section 7.1 in the appendix, the Betti numbers of $\widetilde{F}_{W\setminus W_1}^*(\mathcal L_\psi(h))$ are bounded in terms of that for the Fourier transform of $\widetilde{F}_{W\setminus W_1,!}(\overline{\mathbb{Q}_l})$, and hence are bounded in terms of $W$ and $\widetilde{F}$.}

Since
$$
  S_{1,m,m',\widetilde F,W}
  =
  S_{1,m,m',\widetilde F,W\setminus W_1}
  +
  S_{1,m,m',\widetilde F,W_1},
$$
it remains to estimate the contribution from $W_1$.

Let $Y$ be an irreducible component of $W_1$.  By the hypothesis that
no surface is mapped to a point,
$$
  \dim\overline{\widetilde F(Y)}
$$
is either $2$ or $1$.

Suppose first that $\overline{\widetilde F(Y)}$ is a surface.  Then
$$
  \widetilde F|_Y
$$
is generically finite.  After deleting appropriate lower-dimensional subsets, the same
generic Fourier-vanishing argument applies to the surface image.  The
remaining lower-dimensional pieces are curves or finite sets.  Their
contributions are $O(p)$ or better, by point counting and
one-dimensional exponential-sum estimates.  Thus the total contribution
of $Y$ is
$$
  O(p^{3/2}).
$$

{As explained in the case of $W\backslash W_1$, the constants in the above bound depend only in terms of $W$ and $\widetilde{F}$.}

Suppose next that $\overline{\widetilde F(Y)}$ is a curve.  Then the
generic fibre of
$$
  \widetilde F|_Y
$$
has dimension $1$.  After removing a curve $Y_2\subset Y$, we may
assume that
$$
  \widetilde F:
  Y\setminus Y_2
  \longrightarrow
  \widetilde F(Y\setminus Y_2)
$$
is smooth with one-dimensional fibres and affine image.  For generic
$(m,m')$, the Artin--Schreier sheaf on the image curve has the expected
one-dimensional vanishing bound $O(p^{1/2})$.  The smooth pullback estimate (again, see Sections 7.1 and 7.2 in the appendix) then gives
the required cohomological bound, corresponding to an $O(p^{3/2})$
estimate on $Y\setminus Y_2$.  The deleted curve $Y_2$ contributes
$O(p)$.

Summing over the finitely many irreducible components of $W_1$, we obtain
$$
  \left|
    S_{1,m,m',\widetilde F,W_1}
  \right|
  =
  O(p^{3/2})
$$
for $(m,m')$ outside the exceptional locus produced by the relevant
Fourier stratifications. {Again, constants in the above bound can be bounded in terms of $W$ and $\widetilde{F}$.}  

{Since the exceptional curves appearing in the above are base change to $\mathbb{F}_p$ by universal curves defined over $\mathbb{Z}$, their degrees are bounded by the degrees of these universal curves, which depend only in terms of $W$ and $\widetilde{F}$.} This completes the proof.
\end{proof}

{\bf Uniformity in the family:}  

We now provide additional arguments to show that the implied constants in the above depend only on the degree $d$, and not on the choice of $\phi _1, \phi _2\in \mathbb{Z}[y]$. To this end, we combine the proof of Theorem  \ref{thm:generic-fixed-map} with the following idea in \cite{BonolisKowalskiWoo} (see, in particular, the proof of Theorem 1.13 therein).  To obtain a uniformity Katz-Laumon type estimate for a universal family $\mathcal{V}_b\subset \mathbb{A}^N$ which depends on parameters $b\in \mathcal{B}$ a closed subvariety of $\mathbb{A}^r$ (both defined over $\mathbb{Z}$),  we push the concerned sheaves to the affine space $\mathbb{A}^r\times \mathbb{A}^N $ and perform Fourier transform with the parameter space $\mathcal{B}\times \mathbb{A}^N$. Then by Katz-Laumon theory, there is an exceptional set $\mathcal{E}\subset \mathbb{A}^r\times \mathbb{A}^N $ (over $\mathbb{Z}$, ) such that for $(b,h)\notin \mathcal{E}$, then the needed estimates are satisfied. In particular, there is a $p_0$  depends only on the family $\mathcal{V}$ (over $\mathbb{Z}$), so that if $p>p_0$, and $\mathcal{V}_p$  and $ \mathcal{E}_b:=\mathcal{E} \cap \{b\}\times \mathbb{A}^N$ have the correct dimensions after base change to $\mathbb{F}_p$, then $ \mathcal{E}_b(\mathbb{F}_p)$  is the exceptional set for $\mathcal{V}_b(\mathbb{F}_p)$. By intersection theory, it is clear that the degree of $\mathcal{E}_b(\mathbb{F}_p)$ is bounded in terms of $\mathcal{E}$ and hence in terms of the degree $d$. While the condition that $\mathcal{V}_p$  and $ \mathcal{E}_b:=\mathcal{E} \cap \{b\}\times \mathbb{A}^N$ have the correct dimensions after base change is only expected outside of a hypersurface $\mathcal{B}_1\subset \mathcal{B}$, we can apply again the argument for the family with parameter in $\mathcal{B}_1$. Repeating this procedure a finite number of times will terminate, because the dimension of $\mathcal{B}, \mathcal{B}_1,\ldots $ keeps decreasing, and we obtain a uniform bound for all parameters $b\in \mathcal{B}$ by the above procedure. 
 
In our case, the universal  families (defined over $\mathbb{Z}$) are those parametrizing the closed subvarieties $V_1,\ldots $ of $\mathbb{A}^3$ which appear in the proof of Theorem \ref{thm:generic-fixed-map}, where the parameters are pairs $(\phi _1,\phi _2)\in \mathbb{Z}[y]$ which vanish at $y=0$ and are linearly independent. The family of such pairs may be parametrized by $$
  \bigl(\mathbb A^{d+1}\setminus\mathbb A^d\bigr)
  \times
  \bigl(\mathbb A^1\setminus\{0\}\bigr)
  \times
  \bigl(\mathbb A^d\setminus\{0\}\bigr).
$$
Here $\mathbb A^{d+1}$ parametrizes polynomials of degree at most $d$,
while $\mathbb A^d$ parametrizes polynomials of degree at most $d-1$.
A triple
$$
  (\phi_1,\alpha,P)
$$
parametrizes the pair
$$
  \phi_1
  \qquad \text{and}\qquad
  \phi_2=\alpha\phi_1+P.
$$
To impose
$$
  \phi_1(0)=\phi_2(0)=0,
$$
we add the corresponding two linear equations on the coefficients. This is a finite dimensional algebraic variety, and we write it as a finite union of affine varieties. Then we obtain a finite number of  algebraic families $\mathcal{V}_1,\ldots $ for which the uniformity argument above can be applied. This completes the proof of the generic bound in Lemma \ref{lem:stratified-sqrt-S1}. We just remark here that after choosing the exceptional set  $\mathcal{E}$ as described in the above procedure, when base change to $\mathbb{F}_p$ we also exclude the parameters $b\in \mathcal{B}(\mathbb{F}_p)$ which make the pair $\phi _1,\phi _2$ no longer linear independent in $\mathbb{F}_p[y]$ (besides the other exclusions already mentioned before).

\subsubsection{The crude bound}
\label{subsubsec:crude-bound}

We next establish the uniform estimate
$$
  \left|
    S_\lambda(m,m')
  \right|
  =
  O(p^2)
$$
for all $m,m'\in\mathbb F_p$.

\begin{theorem}
\label{thm:crude-bound}
For every $m,m'\in\mathbb F_p$,
$$
  |S_\lambda(m,m')|
  =
  O_{\deg}(p^2).
$$ {The constant in the above bound depends only on the degree $d$, and not on $p$.}
\end{theorem}

\begin{proof}
For each $u\in\mathbb A^1$, define
$$
  V_u
  =
  \left\{
    (x,y)\in\mathbb A^2:
    \phi_1(x)-\phi_1(y)=u
  \right\}.
$$
This is an affine plane curve of degree at most $d$.

The correlation sum may be written as
\begin{align*}
  S_\lambda(m,m')
  =\,&
  \sum_{u\in\mathbb F_p}
  \sum_{\substack{
    (y_1,y_2)\in V_u(\mathbb F_p)\\
    (y_3,y_4)\in V_{-u}(\mathbb F_p)}}
  \psi\Bigl(
    D(y_4)-m'\phi_2(y_4)+m'\phi_2(y_3)
  \Bigr)
  \\
  &\hspace{8em}\times
  \psi\Bigl(
    -D(y_2)-m\phi_2(y_2)+m\phi_2(y_1)
  \Bigr).
\end{align*}

For each fixed $u$, the four-variable sum factors as
\begin{align*}
  &
  \sum_{(y_3,y_4)\in V_{-u}(\mathbb F_p)}
  \psi\Bigl(
    D(y_4)+m'\phi_2(y_4)-m'\phi_2(y_3)
  \Bigr)
  \\
  &\qquad\times
  \sum_{(y_1,y_2)\in V_u(\mathbb F_p)}
  \psi\Bigl(
    -D(y_2)-m\phi_2(y_2)+m\phi_2(y_1)
  \Bigr).
\end{align*}

Fix $u\in\mathbb F_p$.  Consider first
$$
  \Sigma_{1,u}(m')
  =
  \sum_{(y_3,y_4)\in V_{-u}(\mathbb F_p)}
  \psi\Bigl(
    D(y_4)+m'\phi_2(y_4)-m'\phi_2(y_3)
  \Bigr).
$$
Let
$$
  \Lambda_1
  \subset
  \mathbb A^1(\overline{\mathbb F}_p)
$$
be the set of $u$ for which the phase
$$
  D(y_4)+m'\phi_2(y_4)-m'\phi_2(y_3)
$$
is constant on a geometric irreducible component of $V_{-u}$.

We claim that
$$
  |\Lambda_1|\leq d-1.
$$
The polynomial
$$
  \phi_1(y_3)-\phi_1(y_4)
$$
is indecomposable in the sense required by Stein's irreducibility theorem.
{(This means that $\phi_1(y_3)-\phi_1(y_4)$ is not of the
form $\kappa(\tau(y_3,y_4))$ for a one-variable polynomial $\kappa$ of
degree at least $2$ and a polynomial $\tau$.  This fact follows by taking
the partial derivatives of $\phi_1(y_3)-\phi_1(y_4)$ and observing that no
polynomial of the form $\kappa(\tau(y_3,y_4))$ can have such partial
derivatives.)} We use here the version of Stein's theorem over algebraically closed fields
of arbitrary characteristic; see \cite{lorenzini}.

By Stein's irreducibility theorem, there are at most $d-1$ values
$$
  u\in\overline{\mathbb F}_p
$$
for which
$$
  \phi_1(y_3)-\phi_1(y_4)+u
$$
is reducible.

Suppose now that this polynomial is geometrically irreducible.  Then
$V_{-u}$ is geometrically irreducible.  If the phase were constant on
$V_{-u}$, there would exist
$$
  v\in\overline{\mathbb F}_p
$$
such that
$$
  \phi_1(y_3)-\phi_1(y_4)+u
$$
divides
$$
  D(y_4)+m'\phi_2(y_4)-m'\phi_2(y_3)-v.
$$
On $V_{-u}$, using
$$
  \phi_2=\alpha\phi_1+D,
$$
the latter polynomial may be rewritten as
$$
  D(y_4)
  +m'\phi_2(y_4)
  -m'\phi_2(y_3)-v
  =
  D(y_4)
  +m'\bigl(D(y_4)-D(y_3)\bigr)
  +\alpha m'u-v.
$$
This is a nonconstant polynomial of degree strictly smaller than
$$
  \deg\bigl(
    \phi_1(y_3)-\phi_1(y_4)+u
  \bigr)
  =
  d.
$$
Therefore such a divisibility relation is impossible.  This proves
$$
  |\Lambda_1|\leq d-1.
$$

For
$$
  u\in\mathbb F_p\setminus\Lambda_1,
$$
{the Bombieri--Weil estimate for exponential sums on
algebraic curves (see Section~VI of \cite{bombieri} and Section~7.3 of the
appendix) gives}
$$
  |\Sigma_{1,u}(m')|
  =
  O_{\deg}(p^{1/2}).
$$
Here the relevant curves are geometrically irreducible outside the
exceptional set considered above; after passing to their normalizations,
the required estimate is the standard Bombieri--Weil bound.  We also use
{the fact, valid once $p$ is sufficiently large relative
to the degrees of the objects involved, that the phase is not of
Artin--Schreier form}
$$
  h^p-h+c.
$$
Since all relevant degrees are $<p$, this exceptional case is ruled out
by the usual pole-order argument at infinity. 

Similarly, there exists a set
$$
  \Lambda_2
  \subset
  \mathbb A^1(\overline{\mathbb F}_p),
  \qquad
  |\Lambda_2|\leq d-1,
$$
such that, for
$$
  u\in\mathbb F_p\setminus\Lambda_2,
$$
one has
$$
\left|
  \sum_{(y_1,y_2)\in V_u(\mathbb F_p)}
  \psi\Bigl(
    -D(y_2)-m\phi_2(y_2)+m\phi_2(y_1)
  \Bigr)
\right|
=
O_{\deg}(p^{1/2}).
$$

We split
$$
  S_\lambda(m,m')
  =
  S_{\mathrm{good}}+S_{\mathrm{bad}},
$$
where $S_{\mathrm{good}}$ denotes the contribution from
$$
  u\in
  \mathbb F_p\setminus(\Lambda_1\cup\Lambda_2),
$$
and $S_{\mathrm{bad}}$ denotes the contribution from
$$
  u\in
  \mathbb F_p\cap(\Lambda_1\cup\Lambda_2).
$$

For every good $u$, both one-dimensional factors are
$O_{\deg}(p^{1/2})$.  Their product is therefore $O_{\deg}(p)$.
Summing over at most $p$ possible values of $u$, we obtain
$$
  |S_{\mathrm{good}}|
  =
  O_{\deg}(p^2).
$$

For the bad values of $u$, we use the trivial point-counting estimate.
Since
$$
  |\Lambda_1\cup\Lambda_2|
  \leq
  2(d-1)
  =
  O_{\deg}(1),
$$
there are only boundedly many such values.  For each of them,
$$
  |V_u(\mathbb F_p)|
  =
  O_{\deg}(p)
$$
and
$$
  |V_{-u}(\mathbb F_p)|
  =
  O_{\deg}(p).
$$
Consequently,
$$
  |V_u(\mathbb F_p)\times V_{-u}(\mathbb F_p)|
  =
  O_{\deg}(p^2),
$$
and hence
$$
  |S_{\mathrm{bad}}|
  =
  O_{\deg}(p^2).
$$

Combining the good and bad contributions gives
$$
  |S_\lambda(m,m')|
  =
  O_{\deg}(p^2)
$$
uniformly for all $m,m'\in\mathbb F_p$. {As in the previous Subsubsection, we can see that the constants in the above bounds depend only on the degree $d$, and not on $p$.}
\end{proof}

\begin{remark}  One might hope that the bound in the above theorem holds for every sufficiently regular
morphism $\widetilde F$, for example for every smooth or finite morphism.
This is false.  Consider
$$
  W=\mathbb A^3
$$
and
$$
  \widetilde F(y_1,y_2,y_3)
  =
  (y_1^2,y_2^2,y_3^2).
$$
For $m=m'=0$,
$$
  S_{1,0,0,\widetilde F,W}
  =
  p^2
  \sum_{y_1\in\mathbb F_p}
  \psi_p(y_1^2).
$$
The one-variable sum is a Gauss sum of absolute value $p^{1/2}$, and
therefore
$$
  \left|
    S_{1,0,0,\widetilde F,W}
  \right|
  =
  p^{5/2}.
$$
Thus the $O(p^2)$ estimate must use the particular form of $W$ and of
the phase arising in our problem.
\end{remark}

{
\begin{proof}[Proof of Theorem~\ref{thm:main}]
Lemma~\ref{lem:stratified-sqrt-S1} supplies the two resonant
exponential-sum estimates required in Section~\ref{sec:admissible}, with
constants and degrees of the exceptional loci depending only on the
prescribed degrees.  The conclusion of that section therefore proves Kernel
Estimate~\ref{est:HLY-kernel} for every nonzero Han--Lacey--Yang shift: the
distinct-degree and equal-degree quadratic cases follow from
\cite{HLY}, the equal-degree nonresonant case follows from
Lemma~\ref{lem:HLY-equal-degree-nonresonant}, and the remaining resonant
case follows from Lemma~\ref{lem:admissible-matrix}.  Proposition~\ref{prop:threshold}
then gives a nontrivial polynomial corner whenever
$|A|\geq C p^{2-1/14}$, after first fixing $C$ and then enlarging the
degree-dependent prime threshold if necessary.  This is precisely the
assertion of Theorem~\ref{thm:main}.
\end{proof}
}

\medskip

\section{Appendix: Preliminaries on exponential sums}
\label{sec:exponential-sum-preliminaries}

{We collect here some results on exponential sums used in
Section~6.  We shall only consider the
standard Artin--Schreier sheaf on an affine space or its restriction to a
closed smooth subvariety.}

We first recall some minimal necessary notions from \cite{KatzLaumon}. 

{Let $X$ be a variety over $S$, with structure morphism
$\pi:X\to S$, and let $\mathcal{F}$ be a bounded constructible
$l$-adic complex.  We define the dual $D_{X/S}(\mathcal{F})$ using the
relative dualising complex
$K_{X/S}=\pi^{!}(\overline{\mathbb{Q}}_l)$.  Here $\pi^{!}$ is the right
adjoint of the compactly supported pushforward $\pi_!$.  If $\pi$ is
smooth of relative dimension $k$, then $\pi^{!}=\pi^*[2k](k)$.  More
precisely,
$D_{X/S}(\mathcal{F})=R\mathcal{H}om_X(\mathcal{F},K_{X/S})$, where
$R\mathcal{H}om_X$ denotes the derived Hom functor.}

{We say that $\mathcal{F}$ is reflexive if the morphism
$\mathcal{F}\rightarrow D_{X/S}D_{X/S}(\mathcal{F})$ is an isomorphism.
The definition in \cite{KatzLaumon} requires a stronger condition,
involving all good base changes $S'\rightarrow S$; that condition is
satisfied after restriction to an appropriate dense open subset of $S$.
Here this stronger condition is not needed: we may fix a good base change
$S'\rightarrow S$ for which the estimates for the relevant Fourier
transforms in the cited paper hold.}

We say that $\mathcal{F}$ is perverse over $S$ if it is reflexive and satisfies the following two conditions on support: for all $j$

(A) $\dim _S (Supp \mathcal{H}^j(\mathcal{F}))\leq -j$, 

and 

(B) $\dim _S (Supp \mathcal{H}^j(D_{X/S}(\mathcal{F})))\leq -j$. 

If $f:X\rightarrow Y$ is a smooth morphism of relative dimension $d$,
{$X$ and $Y$ are smooth affine varieties}, and
$\mathcal{F}$ is perverse on $Y$, then $f^*(\mathcal{F})[d]=f^{!}(\mathcal{F})[-d](-d)$ is perverse. 

In applications, we need to use a more general notion: semiperversity, which concerns those $\mathcal{F}$ which satisfy only  condition (A) above. For example, the  (compactly supported) pushforward of a perverse sheaf by a smooth morphism may not be perverse, but only perverse. 

We also collect some fundamental facts about weights which are needed.
{For details, see Expos\'e~VI of \cite{SGA45}.  Here we
assume that $X$ is defined over $\mathbb{F}_q$, where $q$ is a prime
power.}

An $l$-adic sheaf $\mathcal{F}$ is pointwise of weight $n$ if for all
closed points $x\in X$ {(if
$X=\operatorname{Spec}(A)$ is affine, its closed points correspond to the
maximal ideals of $A$)} the eigenvalues of the pullback of the Frobenius
morphism at the stalk $\mathcal{F}_x$ are complex algebraic numbers with
absolute value $q_x^{n/2}$, {where $q_x$ is the
cardinality of the residue field $\kappa(x)$.}

Deligne's Riemann hypothesis \cite{DeligneWeilII} is the result that if $\mathcal{F}$ has pointwise weight $n$, then all eigenvalues of the pullback of the Frobenius morphism on $H^j_c(X_{\overline{\mathbb{F}_q}}, \mathcal{F})$ have absolute values $\leq q^{(n+j)/2}$. {Weil's classical Riemann hypothesis} corresponds to the case $\mathcal{F}=\overline{\mathbb{Q}_l}$. Weights behave well under $f^*, f^{!}$ (under an appropriate shift and Tate twist as above) and $f_{!}$, where $f$ is a smooth morphism between smooth and affine varieties.   

{For further material on perverse sheaves and weights
useful for our purposes, see \cite[Sections~4 and~5.4]{BonolisKowalskiWoo}.}

\subsection{Vanishing dimension, Betti numbers and weights of sheaves under pullback}

{We next describe how vanishing dimension, Betti numbers and weight behave
under pullback by a smooth quasi-projective morphism.  Our main reference
is Expos\'e~I of \cite{SGA45}.}

We recall that a morphism $G:X\longrightarrow Y$ is quasi-projective if it is the composition of  an open immersion $X\hookrightarrow \mathbb{P}^N_Y$ and the projection $\mathbb{P}^N_Y\rightarrow Y$. If $X$ and $Y$ are affine, then every morphism $G$ is automatically quasi-projective. (Indeed, assume that $X$ is a closed subvariety of $\mathbb{A}^N$. Let $\Gamma \subset \mathbb{A}^N\times Y$ be the graph of $G$. Then $X$ is isomorphic to $\Gamma$, and the morphism $G$ is isomorphic to the projection  $\mathbb{A}^N\times Y\rightarrow Y$. $\Gamma$ has an open immersion into $\mathbb{P}^N\times Y$.)

Let $X$ and $Y$ be affine varieties, defined over $\mathbb{F}_q$, where $q$ is a power of a pime number $p$, and 
$$
  G:X\longrightarrow Y
$$
be a smooth morphism {of} relative dimension $r$, {with $\deg (G)<p$},  and let
$\mathcal L$ be an $\ell$-adic sheaf on $Y$.  The Leray spectral
sequence for compactly supported cohomology, together with the projection
formula, has the form
$$
  E_2^{i,j}
  =
  H_c^i
  \bigl(
    Y\otimes \overline{\mathbb{F}_p} ,
    \mathcal L\otimes ^{\mathbb{L}}R^jG_!\overline{\mathbb Q}_\ell
  \bigr)
  \Longrightarrow
  H_c^{i+j}
  \bigl(
    X\otimes \overline{\mathbb{F}_p},G^*\mathcal L
  \bigr).
$$
We note that $R^jG_!\overline{\mathbb Q}_\ell$ is constructible and moreover perverse. Since every geometric fibre of $G$ has dimension $r$, by {the base-change theorem} we obtain
$$
  R^jG_!\overline{\mathbb Q}_\ell=0
  \qquad (j>2r).
$$
Consequently, if there is $k_Y>0$ so that for every
$
  0\leq i\leq2r, 
$
and 
$
  j>k_Y, 
$
we have
$$
  H_c^i
  \bigl(
    Y\otimes \overline{\mathbb{F}_p},
    \mathcal L\otimes ^{\mathbb{L}}R^jG_!\overline{\mathbb Q}_\ell
  \bigr)
  =
  \textcolor{blue}{0,}
$$
then
\begin{equation}
\label{eq:smooth-pullback-vanishing}
  H_c^k(X\otimes \overline{\mathbb{F}_p},G^*\mathcal L)=0
  \qquad (k>k_Y+2r).
\end{equation}

Note that having a bound on vanishing dimension as in the above formula
implies a bound on the weights that occur in estimating exponential sums.
{In the applications in
Section~\ref{subsubsec:generic-bound}, we need this vanishing bound to hold
when $\mathcal{L}$ is an Artin--Schreier sheaf on $Y$ (corresponding to a
generic linear map) and $k_Y=\dim(Y)$.  This follows from the work of Katz
and Laumon, as summarized below.}

{We note that the Betti numbers of $H_c^k(X,G^*\mathcal L)$ is also bounded in terms of that for $H_c^i
  \bigl(
    Y,
    \mathcal L\otimes ^{\mathbb{L}}R^jG_!\overline{\mathbb Q}_\ell
  \bigr)$.}

\subsection{\texorpdfstring{{Katz--Laumon theorem}}{Katz--Laumon theorem}}

We now recall the Fourier-transform theorem of
{Katz and Laumon}~\cite{KatzLaumon}, which allows us to obtain the bound on vanishing dimension mentioned at the end of the previous subsection.  

Let $Y \subset \mathbb{A}^N$ be a smooth affine variety {(defined over $\mathbb{Z}$)} of pure dimension $n$, of bounded
complexity, and let $\mathcal K$ be a constructible complex of bounded
complexity on $Y$.  Assume moreover that $\mathcal{K}$ is perverse. 

{We now base change to $\mathbb{F}_p$ for a large enough prime number $p$ and} consider the family of additive twists
$$
  \mathcal K\otimes ^{\mathbb{L}}\mathcal L_\psi(h),
  \qquad
  h(z)=m_1z_1+\cdots+m_Nz_N.
$$
{Katz and Laumon} (see Sections 2 and 5 in the cited paper) show that the corresponding Fourier transform is perverse, and on a dense open subset of the dual affine
space its cohomology is concentrated in degree $-N$. {By the base-change theorem (see Section~IV of Expos\'e~I of \cite{SGA45}),} the stalk of this Fourier transform at a geometric point $(m_1,\ldots ,m_N)$ is precisely the compactly supported cohomology groups $H_c^i
  \bigl(
    Y\otimes \overline{\mathbb{F}_p},
    \mathcal K\otimes ^{\mathbb{L}}\mathcal L_\psi(h)
  \bigr)$ considered in the previous subsection.  More precisely, outside a proper homogeneous algebraic subset  of the
dual parameter space {(descended from a universal one $\mathcal{E}$ on $\mathbb{Z}$)}, the compactly supported cohomology of the
corresponding additive twist has the expected middle-dimensional
vanishing. This result is extended to the case where $Y$ is not necessarily smooth or closed in $ \mathbb{A}^N$, and also to semiperverse sheaves. However, here we only have an upper bound for the vanishing dimension, and the cohomology may not be concentrated at one degree as in the perverse sheaf case.  {Applying this to
$\mathcal{K}=R^qG_!\overline{\mathbb Q}_\ell[\dim(Y)]$ in the previous
subsection, we obtain the bound needed for the generic estimate in
Section~\ref{subsubsec:generic-bound}.}

This result is  extended in \cite{FouvryKatz} , which further analyzes what to expect for the exponential sum when the parameters are in the exceptional set $\mathcal{E}$. In this paper, we do not need this deeper stratification result.

\subsection{\texorpdfstring{{Bombieri--Weil upper bounds for
exponential sums on algebraic curves}}{Bombieri--Weil upper bounds for
exponential sums on algebraic curves}}

In the crude bound in Section~\ref{subsubsec:crude-bound}, we need a bound
which is valid for all parameters (instead of only generic parameters as
in Section~\ref{subsubsec:generic-bound}). The bound needed is furnished
by the classical Bombieri--Weil bounds (see Theorem 6 in
\cite{bombieri}), which we now recall. {If $X$ is a smooth
affine curve defined over $\mathbb{F}_p$ and
$f:X\rightarrow\mathbb{A}^1$ is a morphism whose restriction to each
geometric component of $X$ is not of the form $h^p-h+c$, for a rational
function $h$ defined over $\overline{\mathbb{F}_p}$ and a constant
$c\in\overline{\mathbb{F}_p}$, then}
we have a bound
$$\sum _{x\in X(\mathbb{F}_p)}e_p(f(x))=O(p^{1/2}),$$
{where the implied constant depends only on the number of
components of $X$ and their Betti numbers.  These quantities are bounded
above in terms of the degree of the curve $X$.}

\bigskip
\bigskip
\bigskip
\bigskip
{\bf Acknowledgements.}  The authors acknowledge the use of LLM models in collecting relevant references for exponential sums as well as helping polish the paper. The second author was supported by NSTC under Grant 111-2115-M-002-010-MY5. A part of this work was done while the third author was visiting the University of Tokyo and Shanghai University of Finance and Economics, and he would like to thank the institutes for support and inspiring environment.   The third author thanks Fabrizio Catanese for providing the reference \cite{milnor}.


\bigskip
\bigskip


























\begin{thebibliography}{99}
\bibitem{BBD}
A.~A. Beilinson, J.~Bernstein  and P.~Deligne,
\emph{Faisceaux pervers}, Ast\'erisque \textbf{100}, 1982.

\bibitem{BergelsonLeibman}
V.~Bergelson and A.~Leibman,
\emph{Polynomial extensions of van der Waerden's and Szemer\'edi's theorems},
J. Amer. Math. Soc. \textbf{9} (1996), no.~3, 725--753.

\bibitem{BonolisKowalskiWoo}
D.~Bonolis, E.~Kowalski  and K.~Woo,
\emph{Stratification theorems for exponential sums in families}, arXiv:2506.18299.

\bibitem{BourgainChang}
J.~Bourgain and M.-C.~Chang,
\emph{Nonlinear Roth type theorems in finite fields},
Israel J. Math. \textbf{221} (2017), no.~2, 853--867.

\bibitem{bombieri}
E.~Bombieri,
\emph{On exponential sums in finite fields},
Amer. J. Math. \textbf{88} (1966), no.~1, 71--105.

\bibitem{SGA45}
P.~Deligne,
\emph{Cohomologie \'etale}, SGA $4\frac12$,
Lecture Notes in Mathematics, vol.~569, Springer, 1977.

\bibitem{DeligneWeilII}
P.~Deligne,
\emph{La conjecture de Weil. II}, Publ. Math. IH\'ES \textbf{52} (1980), 137--252.

\bibitem{DLS}
D.~Dong, X.~Li  and W.~Sawin,
\emph{Improved estimates for polynomial Roth type theorems in finite fields},
J. Anal. Math. \textbf{141} (2020), no.~2, 689--705.



\bibitem{FouvryKatz}
\'E.~Fouvry and N.~M. Katz,
\emph{A general stratification theorem for exponential sums, and applications},
J. Reine Angew. Math. \textbf{540} (2001), 115--166.

\bibitem{FKMConductor}
\'E.~Fouvry, E.~Kowalski and P.~Michel,
\emph{On the conductor of cohomological transforms}, Ann. Fac. Sci. Toulouse Math.
\textbf{30} (2021), no.~1, 203--254.

\bibitem{HLY}
R.~Han, M.~T. Lacey and F.~Yang,
\emph{A polynomial Roth theorem for corners in finite fields},
Mathematika \textbf{67} (2021), no.~4, 885--896.

{
\bibitem{KucaDistinct}
B.~Kuca,
\emph{Multidimensional polynomial Szemer\'edi theorem in finite fields for
polynomials of distinct degrees},
Israel J. Math. \textbf{259} (2024), no.~2, 589--620.

\bibitem{Kuca}
B.~Kuca,
\emph{Multidimensional polynomial patterns over finite fields: bounds,
counting estimates and Gowers norm control},
Adv. Math. \textbf{448} (2024), Paper No.~109700, 61 pp.

\bibitem{Lim}
Z.~L. Lim,
\emph{Uniform nonlinear Szemer\'edi theorem for corners in finite fields},
arXiv:2501.04887 (2025).
}


\bibitem{lorenzini}
D.~Lorenzini,
\emph{Reducibility of polynomials in two variables},
J. Algebra \textbf{156} (1993), 65--75.

\bibitem{QST}
W.~Sawin, A.~Forey, J.~Fres\'an  and E.~Kowalski,
\emph{Quantitative sheaf theory}, J. Amer. Math. Soc. \textbf{36} (2023), no.~3,
653--726.

\bibitem{KatzSingular}
N.~M. Katz,
\emph{Estimates for ``singular'' exponential sums},
Int. Math. Res. Not. \textbf{1999}, no.~16, 875--899.

\bibitem{KatzSums}
N.~M. Katz,
\emph{Sums of Betti numbers in arbitrary characteristic},
Finite Fields Appl. \textbf{7} (2001), no.~1, 29--44.


\bibitem{KatzLaumon}
N.~M. Katz and G.~Laumon,
\emph{Transformation de Fourier et majoration de sommes exponentielles},
Publ. Math. IH\'ES \textbf{62} (1985), 145--202.

\bibitem{Laumon}
G.~Laumon,
\emph{Transformation de Fourier, constantes d'\'equations fonctionnelles et conjecture de Weil},
Publ. Math. IH\'ES \textbf{65} (1987), 131--210.


\bibitem{Peluse}
S.~Peluse,
\emph{Three-term polynomial progressions in subsets of finite fields},
Israel J. Math. \textbf{228} (2018), no.~1, 379--405.

\bibitem{Weil}
A.~Weil,
\emph{On the Riemann hypothesis in function-fields},
Proc. Natl. Acad. Sci. U.S.A. \textbf{27} (1941), 345--347.
\end{thebibliography}
\end{document}